\documentclass[11pt]{article}

\usepackage{amssymb,amsfonts,amsmath,amsthm}
\usepackage{apacite} 
\usepackage{bm}
\usepackage{booktabs}
\usepackage{cases}
\usepackage{csquotes}
\usepackage{cleveref}
\usepackage{comment}
\usepackage{color, colortbl}
\usepackage[font={small},width=.9\textwidth]{caption}
\usepackage{enumerate}
\usepackage{fancyhdr} 
\usepackage[margin=1in]{geometry}
\usepackage{graphicx,rotating}
\usepackage{indentfirst}
\usepackage{longtable}
\usepackage{mathrsfs}
\usepackage{mwe,tikz}
\usepackage{mathtools}
\usepackage[elide]{natbib}
\usepackage[percent]{overpic}
\usepackage{pslatex} 
\usepackage{pdflscape}
\usepackage{picture,xcolor}
\usepackage{setspace}
\usepackage{subcaption}
\usepackage{titlesec}
\usepackage{tocloft}
\usepackage{units}

\usepackage{apalike} 

\usepackage{epsfig}
\usepackage{multicol}
\usepackage{listings}
\usepackage{color,colortbl}

\theoremstyle{plain}
\numberwithin {equation}{section}
\newtheorem{theorem}{Theorem}[section]

\newtheorem{lemma}{Lemma}[section]
\newtheorem{corollary}{Corollary}[section]
\newtheorem{assumption}{Assumption}[section]

\theoremstyle{example}

\theoremstyle{definition}

\theoremstyle{remark}

\newcommand{\cov}{\mathrm{Cov}}

\newcommand{\lam}{\lambda}
\newcommand{\lamn}{\lam^{(n)}}
\newcommand{\Sig}{\bm{\Sigma}}
\newcommand{\X}{\mathbf{X}}
\newcommand{\Y}{\mathbf{Y}}
\newcommand{\tilY}{\tilde{Y}}

\newcommand{\imain}{I_{main}}
\newcommand{\iout}{I_{out}}
\newcommand{\tauif}{\tau_{i,1}}
\newcommand{\tauis}{\tau_{i,2}}
\newcommand{\tauifn}{\tau_{i,1}^{(n)}}
\newcommand{\tauisn}{\tau_{i,2}^{(n)}}

\newcommand{\sqtauis}{\sqrt{\tau_{i,2}}}

\newcommand{\U}{\mathbf{U}}
\newcommand{\Ud}{\U_{d}}
\newcommand{\Lam}{\bm{\Lambda}}
\newcommand{\Lamd}{\Lam_{d}}
\newcommand{\Sigd}{\Sig_{d}}

\newcommand{\clam}{c_{\lambda}}

\newcommand{\hatlam}{\hat{\lambda}}
\newcommand{\hatU}{\hat{U}}
\newcommand{\hatUbf}{\hat{\U}}

\newcommand{\Sbf}{\mathbf{S}}

\newcommand{\delmn}{\delta_{m}^{(n)}}
\newcommand{\deln}{\delta^{(n)}}

\newcommand{\haty}{\hat{y}}

\newcommand{\Z}{\mathbf{Z}}

\newcommand{\samcov}{\hat{\Sig}}
\newcommand{\dual}[1]{#1_{D}}
\newcommand{\eigval}[2]{\lam_{#1}(#2)}
\newcommand{\taun}[1]{\tau_{#1}^{(n)}}

\newcommand{\tilYsub}[2]{\tilde{Y}_{#1,(#2)}}
\newcommand{\tilZsub}[2]{\tilde{Z}_{#1,(#2)}}
\newcommand{\tilZ}{\tilde{Z}}
\newcommand{\A}{\mathbf{A}}
\newcommand{\B}{\mathbf{B}}

\newcommand{\hatu}{\hat{u}}
\newcommand{\hatuki}{\hat{u}_{ki}}
\newcommand{\dn}{\frac{d}{n}}
\newcommand{\rightarrowas}{\rightarrow_{a.s.}}
\newcommand{\delfrac}[2]{\frac{\deln_{#1}}{\deln_{#2}}}
\newcommand{\bigo}[1]{O \Big( #1 \Big)}
\newcommand{\littleo}[1]{o \Big( #1 \Big)}
\newcommand{\matroot}[1]{#1^{\frac{1}{2}}}
\newcommand{\matrootinv}[1]{#1^{-\frac{1}{2}}}
\newcommand{\hatLambf}{\hat{\Lam}}
\newcommand{\sumi}{\sum_{i=1}^{d}}
\newcommand{\sumj}{\sum_{j=1}^{n}}
\newcommand{\meani}{\frac{1}{d}\sum_{i=1}^{d}}
\newcommand{\meanj}{\frac{1}{n}\sum_{j=1}^{n}}
\newcommand{\checkYi}{\breve{Y}_i}

\newcommand{\maxunder}[1]{\underset{#1}{\max}}

\newcommand{\limunder}[1]{\underset{#1}{\lim}}
\newcommand{\checkYbf}{\breve{\Y}}
\newcommand{\checkY}{\breve{Y}}
\newcommand{\tr}{\mathrm{tr}}
\newcommand{\taud}{\tau^{(d)}}

\newcolumntype{P}{>{\raggedright\arraybackslash}p}
\definecolor{LightCyan}{rgb}{0.88,1,1}

\begin{document}
\title{Theory of high-dimensional outliers}
\author{Hyo Young Choi and J. S. Marron\\
Department of Medicine, University of Tennessee Health Science Center \\
Department of Statistics and Operations Research, University of North Carolina at Chapel Hill}
\date{}
\maketitle


\textbf{Abstract.}  
This study concerns the issue of high dimensional outliers which are challenging to distinguish from inliers due to the special structure of high dimensional space. We introduce a new notion of high dimensional outliers that embraces various types and provides deep insights into understanding the behavior of these outliers based on several asymptotic regimes. Our study of geometrical properties of high dimensional outliers reveals an interesting transition phenomenon of outliers from near the surface of a high dimensional sphere to being distant from the sphere. Also, we study the PCA subspace consistency when data contain a limited number of outliers.  
~\\


\thispagestyle{empty}
\bigskip
\pagebreak
\setcounter{page}{1}
\section{Introduction}\label{ch_theory:motivation}
\noindent From a classical point of view, outliers have been considered as \emph{bad} cases that may confound the statistical analysis. In this case, one may think the data are contaminated by a few outliers and those should be down-weighted or potentially removed from the dataset. Much work in this case has been done. See \cite{hampel2011robust} and \cite{huber2011robust} for a good overview. On the other hand, there are situations where outliers can produce important and rich information. For example, aberrant observations of gene expression data can be highly related to important genetic phenomena such as mutations, abnormal splicing, and structural variations that are known to be strongly connected to cancer. In both cases, the study of outliers helps to better understand data. 

Roughly speaking, in low dimensional space, a data point is an outlier if it does not fit the distribution that a majority of the data points come from. However, this definition is more challenging for high dimensional data due to the ‘curse of dimensionality’, i.e. the phenomenon where the data points tend to be more apart from each other as the dimension increases. As discussed in Section \ref{ch_theory:geometry}, when $d \gg n$, \cite{Hall2005} showed that data points tend to lie near the surface of a high-dimensional sphere and that, more surprisingly, all pairwise distances of points are approximately equal and all pairwise angles are approximately perpendicular. These geometrical properties indicate that data points in high dimensional space are very sparse, and thus they might be considered as \textit{inliers}, which makes it challenging to distinguish outliers from them. Due to this curse of dimensionality, classical outlier detection methods such as distance-based or depth-based approaches \citep*{Barnett1974outliers,Hawkins1980,Stahel1981,Donoho1992,Liu1992,Zuo2000,Zuo2003,Dang2010} do not work well for high dimensional data.
 Over the last decade, several alternative outlier detection methods have been developed to tackle the challenge of high dimensionality. \citep{filzmoser2008outlier,ro2015outlier,Rousseeuw2016,ahn2018distance} However, there is no consensus on the definition of outliers and each method targets different types of outliers. In this paper, we introduce a new notion of high dimensional outliers that embraces various types of outliers and provides deep insights into understanding the behaviors of outliers in high dimensions. 

 Often, the classical large sample theory does not provide good approximations to high dimensional data. For example, many statistics such as Hotelling's $T^2$-statistic, generalized variances, multiple correlation coefficients, and various statistics for sphericity tests are asymptotically consistent under a classical asymptotic regime, but those asymptotics are no longer valid with large $d$ even $d < n$. To understand such different asymptotic behavior of high dimensional data, as mentioned earlier, tremendous efforts have been made over the last few decades under several different asymptotic regimes \citep{baik2006eigenvalues,jung2009pca,shen2016general,wang2013sphericity,paul2014random,yao2015large}. However, the studies on limiting properties of high dimensional outliers are still lacking. Under the new notion of outliers, we investigate the conditions under which outliers can be distinguished from inliers as well as the conditions under which such outliers can be asymptotically well captured by a low dimensional subspace produced by PCA. Our theoretical results extend the previous asymptotic studies for high dimensional data to the case where there are a small number of outliers. 

The remainder of this paper is organized as follows. In Section \ref{ch_theory:related_work}, we review related work. Section \ref{ch_theory:model} introduces a model for an underlying distribution possibly generating outliers. Some geometrical properties of high dimensional outliers are explored in Section \ref{ch_theory:geometry}. Theoretical aspects related to the asymptotic behavior of sample eigenvalues and eigenvectors when there are two different types of signals, outlier signals as well as main signals, are investigated in Section \ref{ch_theory:asymptotics}. Section \ref{ch_theory:toy_example} provides a toy example to illustrate the theoretical results. The proofs of the theorems are given in Section \ref{ch_theory:proofs}.

\section{Related work}\label{ch_theory:related_work}
\noindent Let $X$ be a $d$-dimensional random vector with mean vector $\mu$ and covariance matrix $\Sig$. Let $\lam_1 \geq \cdots \geq \lam_d$ be the $d$ ordered eigenvalues of $\Sig$ and $U_1, \cdots, U_d$ be the corresponding eigenvectors. Let $X_1, \cdots, X_n$ be observations on $X$. Denote the sample covariance matrix by $\Sbf = \frac{1}{n-1}\sumj (X_j-\bar{X})(X_j - \bar{X})^{T}$ and its ordered sample eigenvalues and eigenvectors by $\hatlam_1 \geq \cdots \geq \hatlam_d$ and $\hatU_1, \cdots, \hatU_d$, respectively. The asymptotic study of sample eigenvalues ($\hatlam_1 \geq \cdots \geq \hatlam_d$) and sample eigenvectors ($\hatU_1, \cdots, \hatU_d$) has an interesting history and developed roughly in three different asymptotic domains: the classical domain, the random matrix theory (RMT) domain, and the high dimensional low samples size (HDLSS). In each domain, different asymptotic theories have been established.

In the classical domain, \cite{girshick1939sampling} investigated the asymptotic properties of sample eigenvalues and eigenvectors in the case of all the eigenvalues of $\Sig$ being different. When the smallest $d-q$ eigenvalues of $\Sig$ are equal and the others are all different, \cite{lawley1953modified} investigated the asymptotic theories of sample eigenvectors. When $X_1, \cdots, X_n$ are from a multivariate normal distribution, \cite{anderson1963asymptotic} has given the asymptotic distribution of $\hatlam_1, \cdots, \hatlam_d$, $\hatU_1, \cdots, \hatU_d$ in the case of $\lam_1, \cdots, \lam_d$ having any multiplicities. The asymptotic study of the eigenstructure of the sample covariance matrix in the classical domain essentially relies on the fact that the population covariance matrix is well approximated by the sample covariance matrix when the sample size is large with dimension fixed. When the dimension is also large, however, this is no longer the case.  

In the RMT domain, these phenomena were explored in a large number of papers, e.g. \cite{marvcenko1967distribution,silverstein1995analysis,silverstein1995strong,bai1998no,baik2005phase,paul2007asymptotics,bai2012sample}. See also \cite{bai2008methodologies}, \cite{bai2010spectral}, \cite{paul2014random} and \cite{yao2015large} for useful overview. A well-known observation is that the empirical spectral distribution (ESD) of the sample covariance matrix converges almost surely to the Marcenko-Pastur distribution when the population covariance matrix is the identity and $d$ and $n$ proportionally grow to infinity. Combining the fact that the an eigenvalue is a continuous function of a matrix, this supports the idea that the sample covariance matrix is not a good estimate of the population covariance matrix for large dimensions. However, many data sets in high dimensions involve quiet different eigenvalues, for instance, a few largest of those are much larger than the other eigenvalues. To understand these phenomena, the \emph{spiked covariance model} was initially introduced by \cite{johnstone2001distribution} and extensively studied. \cite{baik2005phase} studied the conditions of the first $m$ population eigenvalues that provided the corresponding sample eigenvalues being separate from the other small eigenvalues under the spike covariance model. They proved a transition phenomenon: the limits of the extreme sample eigenvalues depend on the critical value $1+\sqrt{c}$, i.e. a  sample eigenvalue from a population eigenvalue that is greater than $1+\sqrt{c}$ is asymptotically isolated from the others, i.e. the \emph{bulk} eigenvalues. \cite{baik2006eigenvalues} extended the results of \cite{baik2005phase} to non-Gaussian variables and found that the limits of the extreme sample eigenvalues depend on the critical values $1+\sqrt{c}$ for the largest spike eigenvalues and on $1-\sqrt{c}$ for the smallest spike eigenvalues. \cite{bai2012sample} extended the results to a generalized spike covariance model that allows flexibility on the distribution of bulk population eigenvalues. The spike covariance model is closely related to the concept of small-rank perturbations, i.e. theories on perturbed random matrices. In a small-rank perturbation approach, convergence of the few largest sample eigenvalues and the corresponding sample eigenvectors are studied in \cite{benaych2011eigenvalues}. 

Note that underlying spike eigenvalues are constant in the classical domain and the RMT domain where the increasing sample size $n$ boosts the consistency. On the other hand, in the HDLSS domain, underlying spike eigenvalues are allowed to increase, which encourages the PCA consistency for increasing dimension $d$ and a fixed $n$ \citep{ahn2007high,jung2009pca,jung2012boundary,shen2016general}. \cite{jung2009pca} explored the asymptotic behaviors of the spike eigenvectors when the levels of spike eigenvalues increase at the rate $d^{\alpha}$. In the case of $\alpha > 1$, they showed that the spike eigenvectors are subspace consistent, i.e. the subspace spanned by the sample spike eigenvectors consistently estimates the subspace spanned by the underlying population spike eigenvectors, and is strongly inconsistent for $\alpha<1$, i.e. the angle between each sample eigenvector and the true one converges to 90 degrees. \cite{jung2012boundary} deeply explored the boundary case ($\alpha=1$) and showed the convergence in distribution of the first spike eigenvector under the normal assumption. \cite{shen2016general} have provided a general framework of the PCA consistency that nicely connected the existing results from different domains except for some boundary cases. 

In this paper, we deeply explore the behaviors of high dimensional outliers via geometrical representations in the HDLSS domain and asymptotic theories of sample eigenvalues and eigenvectors from the data containing a few outliers under the general framework studied in \cite{shen2016general}. A major interest is the consistent estimation of underlying outlier directions in which only a small number of outliers go. We will provide for each scenario a condition that allows achievement of the PCA individual consistency or subspace consistency.


\section{Model and Notations}\label{ch_theory:model}
\noindent In this section, we introduce a model that provides a new notion of high dimensional outliers. Figure \ref{fig:RNAseq_TP53} shows a motivating example with 30 normal RNA-seq data curves in grey color with two potential colored outliers. Each curve represents each observation in the genomic region around the gene TP53. The curves are read depth (or coverage), i.e. the number of reads aligned to each nucleoide, which are log10 read counts from RNA-seq experiments based on HNSCC (Head and Neck Squamous Cell Carcinoma) cancer tissue samples obtained from the TCGA Research Network. We use the terminology \emph{sample} to indicate a patient. Exons, highlighted by colored background (except for pink), are regions of a gene that are annotated as the part of the messenger RNA region. By contrast, introns, highlighted by a white background together with on colored pink, are regions of the gene that are expected to be \emph{spliced out}, i.e. not used in the RNA production. In the figure, the red and blue curves behave differently from the others in the sense that the red one retains an intron, as highlighted by the pink background, and the blue one skips several exons, as highlighted by the green background. Such abnormal splicing events are called \textit{intron retention} and \textit{exon skipping}, respectively. It has been observed that such events happen with a small chance at each gene, hence it makes sense to consider such samples as outliers.

The two red and blue outliers show clearly different structure from the other curves, which implies that they show different underlying signals that do not fit together with the majority of the data. At the same time, interestingly, some of the main structures of the two outliers are shared with most of the data. This example motivated us to consider two different types of underlying directions in the data space, together with variation in those directions in describing outliers. Two important types are \textit{outlier directions} that may lead to prominent high dimensional outliers and \textit{main directions} whose variation is shared among all data points including outliers. The new proposed model incorporating these two components is now introduced in three parts.

\begin{figure}[t]	\centering
	\includegraphics[scale=0.6]{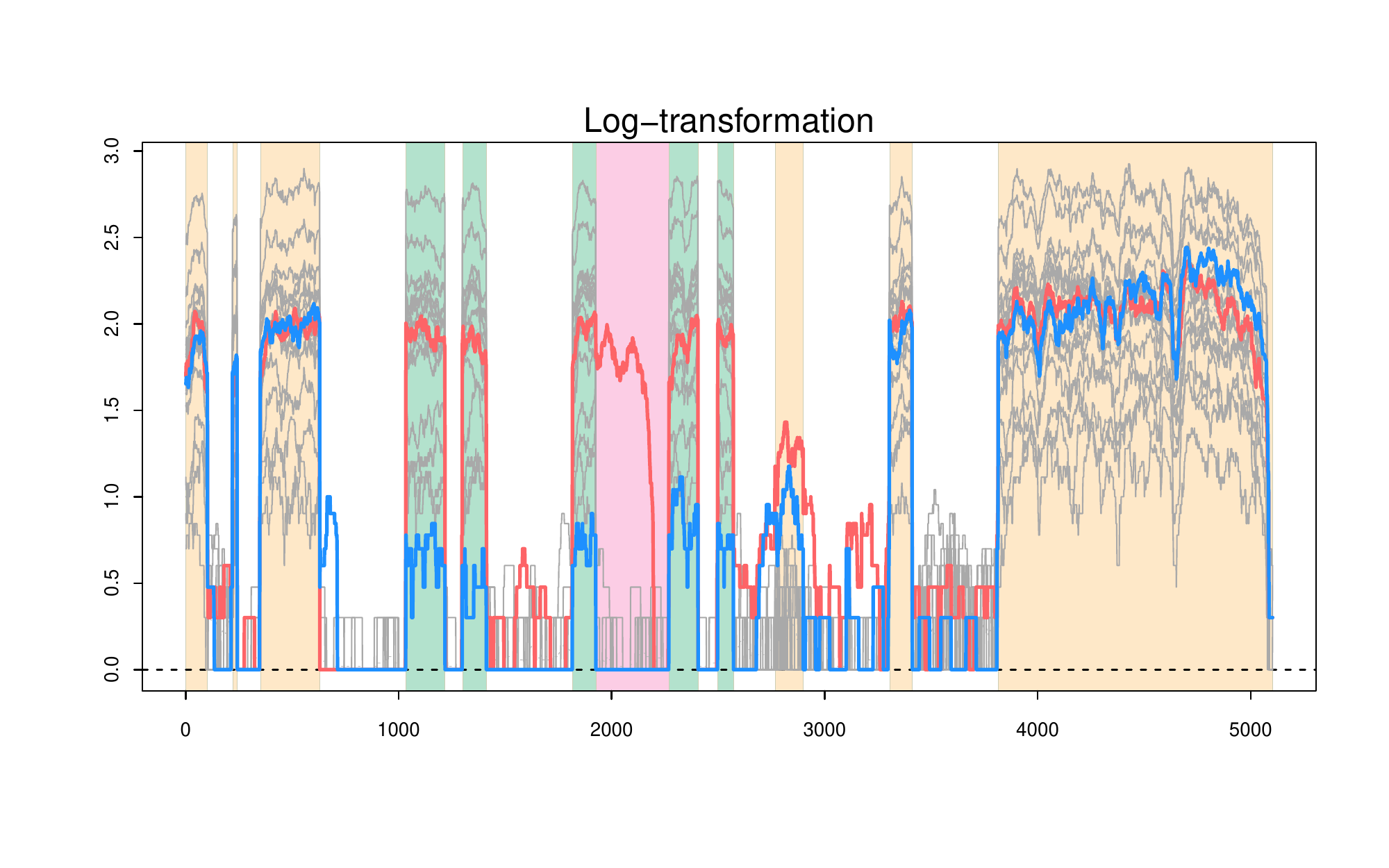}
	\caption[TP53 RNA-seq data]{The 30 RNA-seq observations for the gene TP53 are plotted on the log-scale. Exons are highlighted by colored background and introns are indicated using mostly white background. The red and blue curves indicate biologically important outliers and the other gray curves indicate normal observations.}
\label{fig:RNAseq_TP53}
\end{figure}

Part 1. The classical way of describing underlying variations of a random vector using PCA is discussed in this paragraph. Let $X$ be a random vector distributed as a $d$-dimensional multivariate normal distribution, $N_d (0, \Sigd)$. The spectral decomposition of the population covariance matrix is 
\begin{eqnarray*}
	\Sig_d = \Ud \Lamd \Ud^{T}
\end{eqnarray*} 
where $\Ud = [u_1, \cdots, u_d]$ contains the orthonormal eigenvectors of $\Sigd$ in its columns and $\Lamd = \mathrm{diag}(\lam_1, \lam_2, \cdots, \lam_d)$ is a diagonal matrix with the corresponding non-negative eigenvalues. Then, a random vector from $N_d (0, \Sigd)$ can be expressed as 
\begin{eqnarray*}
	X = \Ud \Lamd^{1/2}Z = \Ud Y
\end{eqnarray*}
where $Z \sim N_d (0, I_d )$ and $Y \sim N_d (0, \Lamd )$. That is, $X$ is a linear combination of $U_i$ with random coefficients $y_i$ from $N(0, \lam_i )$, i.e., 
\begin{eqnarray}\label{eq:linearsum}
	X= \sum_{i=1}^{d} y_i U_i,~~~~y_i \sim N(0, \lam_i ).
\end{eqnarray} 
In the terminology of PCA, the $y_i$ are the \textit{principal components}, i.e. the scores or projection coefficients \citep*{jolliffe2002}. Intuitively, if $X$ involves large $y_i$ for some $i$, then the direction $U_i$ is an important direction of variation of the underlying distribution of $X$, whereas if $y_i \approx 0$, $X$ does not feel strongly the direction $U_i$. 

Part 2. Distributions for modeling outliers are now considered. Based on the intuition behind the principal components, an outlier can be viewed as an observation that goes strongly in some directions that the bulk of data points do not. Denote one of those directions by $U_i^*$ and the corresponding random coefficient by $y_{i}^{*}$. Then outliers that go in the direction $U_i^*$ have large $y_i^*$'s whereas the other data points have small $y_i^*$'s in (\ref{eq:linearsum}). We model this underlying variation of a random coefficient $y_i^*$ by a scale mixture distribution with two different variances, $\tauis \gg  \tauif > 0$, i.e., 
\begin{eqnarray}\label{eq:mixture0}
	y_i^* \sim \left\{ \begin{array}{l l}
					\sqrt{\tauif} z_{i}, & ~~\mathrm{w.p.}~~1-w_{i} \\ 
					\sqtauis z_{i}, & ~~\mathrm{w.p.}~~ w_{i},
				    \end{array} \right.
\end{eqnarray}
where the $z_i$'s are i.i.d random variables with mean zero and variance one and $ 0 \leq w_i \leq 1$, with $w_i \approx 0$. The first part of the mixture distribution with the smaller variance, $\tauif$, describes the behavior of the majority of data vectors with little variation in the direction $U_i^*$. The second part of the mixture distribution with the larger variance, $\tauis$, corresponds to outliers, and so we assume that $w_i$ is small, e.g. less than $0.05$. This mixture model well reflects an underlying mechanism generating outliers in the sense that \enquote{one person's noise could be another person's signal}, as pointed out in \cite{kamber2001data}. 

Part 3. A new model for an underlying distribution embracing a small set of outliers is introduced based on the classical setting (\ref{eq:linearsum}) together with the distribution (\ref{eq:mixture0}) beyond the Gaussian models. Let $\X = [X_1, \cdots, X_n ]$ be a data matrix whose columns are independent  observation vectors distributed as a $d$-dimensional (perhaps non-Gaussian) multivariate distribution with a small number of aberrant vectors whose signals are different from the majority of the data. Let $\{U_i\}_{1\leq i \leq d}$ be a set of underlying orthogonal vectors some of which are responsible for the potential outliers. Note that these vectors do not need to be the eigenvectors of the underlying covariance matrix. In the spirit of (\ref{eq:linearsum}), an observation vector $X_j$ can be expressed as a linear combination of the orthonormal direction vectors, $\{U_i\}_{1\leq i \leq d}$, whose coefficients are independent random variables distributed as different mixture distributions, i.e.
\begin{eqnarray}\label{eq:mixture}
X_j = \sum_{i=1}^{d} y_{ij}U_i, \mathrm{~~where~~}	y_{ij} \sim \left\{ \begin{array}{l l}
					\sqrt{\tauif} z_{ij}, & ~~\mathrm{w.p.}~~1-w_{i} \\ 
					\sqtauis z_{ij}, & ~~\mathrm{w.p.}~~ w_{i},
				    \end{array} \right.
\end{eqnarray}
where the $z_{ij}$'s are assumed to be i.i.d. random variables with mean zero, variance one, and bounded fourth moment. Then, the random variables $\{y_{ij}\}_{1 \leq j \leq n}$ with $w_i > 0$ model how the direction $U_i$ as an outlier component can generate outliers. Also, we will use $w_i = 0$ for other directions especially main components, which allows flexibility to include the classical way of describing the variation from underlying directions as in (\ref{eq:linearsum}). To distinguish the two components, we let $\imain$ denote a set of dimension indices that correspond to main components and $\iout$ denote outlier components. That is, $\iout = \{1 \leq i \leq d ~|~ w_i > 0\}$ and $\imain = \{ 1 \leq i \leq d ~|~ w_i=0 \} = \{ 1, \cdots, d\}\backslash \iout $. Also, we denote the sample indices that are outlying in each outlier component, indexed by $i \in \iout$, by $s_{i}=\{1 \leq j \leq n ~|~ y_{ij}=\sqrt{\tauis}z_{ij} \}$. 

Under the model (\ref{eq:mixture}), outliers are allowed to share important features or background noise with normal data points. The model also allows an outlier to be associated with several outlier components, which offers flexibility in modeling the nature of outliers. Under this setting, a sample vector from (\ref{eq:mixture}) can be viewed as a random vector from a complicated mixture distribution whose components have different covariance structures. 

As discussed earlier, still there is no consensus definition for outliers. Every procedure may target its own informal definition for outliers based on various goals. Here, we describe several types of outliers that are commonly used in various applications as special cases of the proposed model in (\ref{eq:mixture}).
\begin{itemize}
    \item Variable-specific outliers: This type of outlier is different from the bulk of the data only at single variables. If an observation is an outlier with respect to the original variables, then it is usually extreme on these variables. Assuming there are $d$ variables in the model, each sample can be modeled by (\ref{eq:mixture}) with $U_i=e_i$ for $i=1, \cdots, d$. Here, $e_i$ is a unit vector with 1 for the $i$th entry and 0 for the others. Then, an outlier $X_j$ in the $m$-th variable can be described by $\tau_{m,2}>\tau_{m,1}$ and an underlying outlier proportion $w_m$.
    \item Scale mixture outliers: The outliers in this category exhibit a much different abberation, across all variables simultaneously, and are more scattered than the majority of data, and thus they are also known as \emph{scatter outliers} \citep{filzmoser2008outlier}. Let $X_j \sim N_d(0, \sigma_{1}^2 \Sig)$ with probability $1-p$ and $ N_d(0, \sigma_{2}^2 \Sig)$ with a small probability $p$ and $\sigma_2^2 \gg \sigma_1^2$. This scale mixture model is a special subset of the model (\ref{eq:mixture}) with $w_i=p, \tauif = \sigma_1^2, \tauis = \sigma_2^2$ for all $i=1, \cdots, d$, where the $U_i$'s are a set of the orthogonal vectors, e.g. the eigenvectors of $\Sig$. Additionally, the $\iout$ will include every index, $\imain=\emptyset$, and $s_1 = s_2 = \cdots = s_d$. That is, 
    \begin{eqnarray*}
        X_j = \left\{ \begin{array}{l l}
                                    \sum_{i=1}^{d} y_{ij}U_i, \mathrm{~~where~~}	y_{ij}=\sigma_1 z_{ij}, &~~\mathrm{w.p.}~~1-p \\
                                    \sum_{i=1}^{d} y_{ij}U_i, \mathrm{~~where~~}	y_{ij}=\sigma_2 z_{ij}, &~~\mathrm{w.p.}~~p.
                                \end{array} \right.
    \end{eqnarray*}
    \item Shifted outliers: The shifted outliers are those that are shifted globally to a common direction \citep{filzmoser2008outlier,ro2015outlier,dai2016directional}. Often, these outliers share most of the variation with the the bulk of the data, but present abnormally high or low overall pattern, which is typically described by the mean vector denoted by $\mu$. Let $X_j$ be independent random vectors from $a_j \mu + Z_j$, where $Z_j \sim N_d(0, \Sigma)$, $a_j \sim N(0, \sigma_1^2)$ with probability $1-p$ and $N(0, \sigma_2^2)$ with probability $p$, and $Z_j$ and $a_j$ are independent. Assuming $\sigma_1 < \sigma_2$ and a small $p$, the random variable $a_j$ describes how a small fraction of data points may be shifted. Define one of the underlying vectors, say $U_1$, to be the normalized mean vector, that is, $U_1=\mu/\Vert \mu \Vert$, and the other underlying vectors to be orthogonal to each other. Then, the variation from the $U_1$ for normal samples and outliers are respectively $\sigma_1^2 \Vert \mu \Vert^2 + U_1^{T}\Sig U_1$ and $\sigma_2^2 \Vert \mu \Vert^2 + U_1^{T}\Sig U_1$. Thus, each data object can be modeled by 
      \begin{eqnarray*}
            X_j = \sum_{i=1}^{d} y_{ij}U_i, \mathrm{~~where~~}	y_{1j} \sim \left\{ \begin{array}{l l}
					N(0, \sigma_1^2 \Vert \mu \Vert^2 + U_1^{T}\Sig U_1) , & ~~\mathrm{w.p.}~~1-p \\ 
					N(0, \sigma_2^2 \Vert \mu \Vert^2 + U_1^{T}\Sig U_1) , & ~~\mathrm{w.p.}~~ p,
				    \end{array} \right.
    \end{eqnarray*}
    with the other $y_{ij}$ from $N(0, U_i^{T}\Sig U_i)$ for $i=2, \cdots, d$.
    \end{itemize}

\section{Geometrical representation}\label{ch_theory:geometry}
\noindent It is important to understand the behavior of outliers in high dimensional space. Roughly speaking, the distance between data points becomes heavily dominated by noise as dimension increases, resulting in a sparse data set where outliers are less distinguishable. \cite{zhou2016visualization} studied the case where some outliers are too close to each other due to some common factors, e.g. family members, and thus unduly affect the conventional PCA and some robust methods. However, studies on the behavior of high dimensional outliers in a systematic manner are still lacking. This section explores the geometrical features of the high dimensional outliers based on the model (\ref{eq:mixture}). 

It is of great interest to understand when outliers in high dimensions may deviate from the majority and when they may not. Intuitively, if $\tauis$ in some outlier components are dramatically larger than $\tauif$, then the relevant outliers are more likely to be separated from the bulk of the data. By contrast, if $\tauis$ do not differ much from $\tauif$, the corresponding outliers are expected to be harder to distinguish. As discussed below, an interesting observation in high dimensional data is that if an outlier is involved in a large fraction of outlier directions, $d$ encourages the separability of the outlier from the other normal data points even when $\tauis$ is not substantially large. On the other hand, if an outlier is involved in a limited number of outlier directions, $d$ discourages the separability even for relatively large $\tauis$'s. We study these phenomena using the geometrical representation of high dimensional outliers in the HDLSS context explored by \cite{Hall2005} and identify a condition when outliers may be distinguishable in such high dimensions. 

We consider a simple scenario where data come from (\ref{eq:mixture}) with $\tauif=\sigma^2$ and $\tauis=\taud$ for all $i$ under the normality assumption. In this section, we index the variation for outlier components by $d$, $\taud$, as an indication of increase with dimension. Then, our model can be expressed as 
\begin{eqnarray}\label{eq:simple_scenario}
		y_{ij} = \left\{ \begin{array}{l l}
					\sigma z_{ij}, & \mathrm{w.p.}~~1-w_{i} \\ 
					\taud z_{ij}, & \mathrm{w.p.}~~ w_{i},
				    \end{array} \right. ~\mathrm{for}~~i \in \iout~~\mathrm{and}~~y_{ij} = \sigma z_{ij}~~\mathrm{for}~i \notin \iout.
\end{eqnarray} 
Consider a non-outlier point $X_j$ from (\ref{eq:simple_scenario}) which can be expressed as $X_j = \sum_{i=1}^{d} \sigma z_{ij} U_i$ where the $U_i$ are orthonormal underlying eigenvectors. As $d$ increases, it follows by a law of large numbers that its squared Euclidean distance scaled by $d$ converges to the constant $\sigma^2$ in the sense that
\begin{eqnarray}\label{eq:non-outlier-norm}
	\frac{1}{d} \Vert X_j \Vert^2 & = & \frac{1}{d} \sum_{i=1}^{d} \sigma^2 z_{ij}^2 \nonumber \\
		& \rightarrow & \sigma^2
\end{eqnarray}
almost surely. Then, we might fairly say that a non-outlier point $X_j$ lies approximately on the surface of a $d$-variate sphere, of radius $(\sigma^2 d)^{1/2}$, as $d \rightarrow \infty$. Similarly, we can obtain limiting behavior of distances between pairs of non-outlier points. The  distance between two non-outlier points $X_{j}$ and $X_{l}$ is approximately equal to $(2\sigma^2 d)^{1/2}$ as $d \rightarrow \infty$, in the sense that 
\begin{eqnarray}\label{eq:non-outlier-dist}
	\frac{1}{d} \Vert X_j - X_l \Vert^2 & = & \frac{1}{d} \sum_{i=1}^{d} \sigma^2 (z_{ij}- z_{il})^2 \nonumber \\
		& \rightarrow & 2\sigma^2
\end{eqnarray}
where the convergence is almost sure. These asymptotic results match with the results in \cite{Hall2005}. As described in their paper, application of (\ref{eq:non-outlier-dist}) to each pair $(j, l)$ of non-outliers, and scaling all distances by the factor $d^{-1/2}$, shows that they asymptotically construct a polyhedron where each edge is of length $(2\sigma^2)^{1/2}$ and the vertices are the $m$ non-outliers.

Similarly, we now explore the behavior of outliers in high dimensions. An outlier point $X_{j'}$ can be expressed as 
\begin{eqnarray*}
    X_{j'} = \sum_{i\in \iout^{j'}}\sqrt{\taud}z_{ij'}U_i + \sum_{i\notin \iout^{j'}}\sigma z_{ij'}U_i     
\end{eqnarray*}
where $\iout^{j'}$ is an index set for outlier components related to $X_{j'}$. Let $K_{j'}^{(d)} = |\iout^{j'}|$ be the cardinality of the set $\iout^{j'}$ for each $d$ and $p_{out}^{j'} = \limunder{d \rightarrow \infty} \frac{K_{j'}^{(d)}}{d}$ be the fraction of the outliers components for a large $d$. The deviation of $X_{j'}$ from the majority depends on the levels of $K_{j'}^{(d)}$, that is, $p_{out}^{j'}>0$, $p_{out}^{j'}=0$ with $K_{j'}^{(d)} \rightarrow \infty$, and $p_{out}^{j'}=0$ with $K_{j'}^{(d)}$ fixed. Each case requires the different levels of $\taud$ as will be discussed below. 
 
Let us first consider the case of $p_{out}^{j'}>0$ with $\tau = \limunder{d \rightarrow \infty} \taud$. It follows that if a law of large numbers applies to its squared distance divided by $d$, then
\begin{eqnarray}\label{eq:outlier-norm}
	\frac{1}{d} \Vert X_{j'} \Vert^2 & = &  \frac{1}{d}\sum_{i\in \iout^{j'}}\taud z_{ij'}^2 + \frac{1}{d}\sum_{i\notin \iout^{j'}}\sigma^2 z_{ij'}^2 \nonumber \\
	& = & \frac{K_{j'}^{(d)}}{d}\frac{1}{K_{j'}^{(d)}}\sum_{i\in \iout^{j'}} \taud z_{ij'}^2 + \frac{d-K_{j'}^{(d)}}{d}\frac{1}{d-K_{j'}^{(d)}}\sum_{i\notin \iout^{j'}}\sigma^2 z_{ij'}^2 \nonumber \\
	& \rightarrow & p_{out}^{j'} \tau + (1-p_{out}^{j'})\sigma^2
\end{eqnarray}
almost surely as $d \rightarrow \infty$. This implies that an outlier point $X_{j'}$ is approximately of distance $(\sigma^2 d + p_{out}^{j'}(\tau - \sigma^2)d)^{1/2}$ from the origin. Also, the distance between an outlier $X_{j'}$ and a non-outlier $X_{j}$ divided by $d^{1/2}$ converges almost surely to $(p_{out}^{j'} (\tau-\sigma^2) + 2\sigma^2)^{1/2}$ as $d \rightarrow \infty$:
\begin{eqnarray}\label{eq:outlier-dist}
	\frac{1}{d} \Vert X_{j} - X_{j'} \Vert^2 & = & \frac{1}{d} \sum_{i\in \iout^{j'}}(\sigma z_{ij} - \sqrt{\taud}z_{ij'})^2 +  \frac{1}{d}\sum_{i\notin \iout^{j'}}\sigma^2 (z_{ij} - z_{ij'})^2 \nonumber \\
		& \rightarrow & p_{out}^{j'} (\tau-\sigma^2) + 2\sigma^2.
\end{eqnarray}
Therefore, a larger $p_{out}^{j'}$ or a larger $\tau$ help to better separate the outlier $X_{j'}$ from non-outliers provided that $\tau>\sigma^2$ and $p_{out}>0$. In particular, this geometrical property shows that even when $\tau$ is not much bigger than $\sigma^2$, good separability still follows when $p_{out}^{j'}$ is sufficiently large for high dimensions whereas it tends to be less successful in low dimensions \citep{filzmoser2008outlier}.

The type of scale mixture outliers introduced in Section \ref{ch_theory:model} is a special example of this case with $\sigma_1^2 = \sigma^2$ and $\sigma_2^2 = \tau$. For this particular type, all the resulting outliers have $p_{out}=1$, which together with (\ref{eq:non-outlier-norm}) and (\ref{eq:outlier-norm}) leads to two $d$-variate spheres of different radii: a sphere of radius $(\sigma^2 d)^{1/2}$ on the surface of which the non-outliers approximately lie, and another sphere of radius $(\tau d)^{1/2}$ for the outliers. This geometrical representation is also associated with the unique spectrum limit of the sample covariance matrix of high dimensional scale mixture distributions as studied in \cite{li2018structure}. They showed that the limit of the ESD from the scale mixture distribution can be viewed as a mix of the two separate ESD limits relevant to each mixture component, and the separation of these two limits becomes more distinct for a larger ratio of $\frac{d}{n}$. Roughly speaking, the part of the spectrum limit containing large eigenvalues is associated with the larger sphere of radius $(\tau d)^{1/2}$ and the other part involving smaller eigenvalues is associated with a smaller sphere of radius $(\sigma^2 d)^{1/2}$. 
 
So far, we have observed that $d$ encourages the geometrical separability of an outlier $X_{j'}$ if $p_{out}^{j'}>0$. However, this is no longer the case for $p_{out}^{j'}=0$ because the terms $p_{out}^{j'} \tau$ and $p_{out}^{j'}(\tau-\sigma^2)$ in  (\ref{eq:outlier-norm}) and (\ref{eq:outlier-dist}), respectively, disappear for large $d$, which discourages the separability. In this more challenging situation, we need a $\taud$ much bigger than $\sigma^2$ to approximately models the separability. So here we let $\taud$ increase as $d$ increases.  As mentioned earlier, the case with $p_{out}^{j'}=0$ is further divided into two cases where $K_{j'}^{(d)}$ increases as $d$ increases and where $K_{j'}^{(d)}$ is fixed. Let us first explore the case with increasing $K_{j'}^{(d)}$. We model the idea of a stronger outlier as
\begin{eqnarray}\label{eq:geometry_condition}
	\frac{K_{j'}^{(d)}\taud}{d} \rightarrow r_{j'}~~~~\mathrm{as}~d \rightarrow \infty.
\end{eqnarray}
Then, it is easy to show $\frac{1}{d} \Vert X_{j'} \Vert^2 \rightarrow  r_{j'} + \sigma^2$ and $\frac{1}{d} \Vert X_{j} - X_{j'} \Vert^2  \rightarrow  r_{j'} + 2\sigma^2$ as $d \rightarrow \infty$. This indicates that $r_{j'}$ plays an important role in separating $X_{j'}$ from non-outliers geometrically. If $r_{j'}$ is too small, and in particular if it equals $0$, then the data points in the sample including outliers asymptotically behave as a regular data set with the absence of outliers. On the other hand, if the $r_{j'}$ is large enough, the outlier $X_{j'}$ tends to be distinguished from the sphere on the surface of which the majority of data points spread out. 

The results above hold for increasing $K_{j'}^{(d)}$ as $d \rightarrow \infty$, for fixed sample size $n$. For the case of a limited number of outlier directions, i.e. $K_{j'}^{(d)}=K_{j'}$, a law of large numbers may not be applicable, and rather we employ the convergence in distribution. Then, we have $ \frac{1}{d} \Vert X_{j'} \Vert^2 \rightarrow_{d} \frac{1}{K_{j'}}\sum_{i \in \iout^{j'}} r_{j'} z_{ij'}^2 + \sigma^2$ and $\frac{1}{d} \Vert X_{j} - X_{j'} \Vert^2 \rightarrow_{d} \frac{1}{K_{j'}}\sum_{i \in \iout^{j'}} r_{j'} z_{ij'}^2 + 2\sigma^2$. Still, we see that the level of $r_{j'}$ determines the separability of an outlier from the other normal data points. But here it is good to mention that $r_{j'}$ becomes the limit of $\frac{\taud}{d}$, which only depends on the level of $\taud$, because we fix the $K_{j'}$.

To sum up, our study in this section enables understanding of the transition phenomenon of high dimensional outliers from near the surface of a high dimensional sphere to being distant from the sphere. Our results indicate that there are two factors affecting this transition which are the proportion of outlier components involved in an outlier and the signals of those outlier directions. 

\section{PCA consistency}\label{ch_theory:asymptotics}

\noindent In a spike covariance model, a fixed number of population eigenvalues are assumed to be much larger than the others. This provides an important sense in which the signals corresponding to large population eigenvalues are consistently estimated by PCA under some conditions that depend on various asymptotic domains \citep*{shen2016general}. We employ the same concept of a spike covariance model here. Let $K$ be the total number of different spike components among the covariance matrices in the mixture components. For convenience, we refer to $\{U_i \}_{1 \leq i \leq K}$ as \textit{spike directions} and $\{ U_i \}_{K+1 \leq i \leq d}$ as \textit{non-spike directions}. The non-spike components are often considered as noise. In a modification of the definition in Section \ref{ch_theory:model}, denote the index sets for outlier spike components and main spike components by $\iout = \{ 1 \leq i \leq K | w_i > 0\}$ and $\imain = \{1, \cdots , K \} \backslash \iout $, respectively. That is, $\{ U_i \}_{i \in \iout}$ is the set of outlier spike directions and $\{ U_i \}_{i \in \imain}$ is the set of main spike directions. The inherent variation derived in each direction $U_i$ can be expressed as $\lam_i = (1-w_i)\tauif + w_i \tauis $ by the mixture distribution in (\ref{eq:mixture}) and such $\lam_i$'s are indeed the population eigenvalues corresponding to the direction $U_i$. This is because the covariance matrix of $X_j$ from (\ref{eq:mixture}) can be written as $\Sig = \cov(X_j) = \cov(\U y_j) = \U \cov(y_j) \U^{T}$ where $y_j = (y_{1j}, \cdots, y_{dj})^{T}$. Due to the independence of $\{y_{ij}\}_{1 \leq i \leq d}$, $\cov(y_j)$ is a diagonal matrix whose entries are $var(y_{ij})=(1-w_i)\tauif + w_i \tauis$, and thus the $\lam_i$'s are the eigenvalues of $\Sig$ by the eigenvalue decomposition.

Let $X_1, \cdots, X_n$ be observations from (\ref{eq:mixture}) with the $K$ spike components as described above. Denote the sample covariance matrix by $\hat{\Sig} = \frac{1}{n}\X \X^{T} = \frac{1}{n}\sum_{j=1}^{n}X_j X_j^{T}$ and its eigenvalue decomposition by $\hat{\Sig}=\hatUbf \hatLambf \hatUbf^{T} $ with $\hatUbf=[\hatU_1, \cdots, \hatU_d]$ and $\hatLambf=\mathrm{diag}(\hatlam_1, \cdots, \hatlam_d)$ where $\{ (\hatlam_k , \hatU_k ): k=1, \cdots, d \}$ are the pairs of eigenvalues and eigenvectors of $\hat{\Sig}$ such that $\hatlam_1 \geq \hatlam_2 \geq \cdots \geq \hatlam_d $. In this section, asymptotic properties of $\hatlam_1, \cdots, \hatlam_d$ and $\hatU_1, \cdots, \hatU_d$ are analyzed under the general framework developed by \cite{shen2016general}. As discussed in Section \ref{ch_theory:related_work}, this general framework includes several previously studied domains as special cases and allows one to understand interesting connections among the various domains. This section provides parallel asymptotic results for data from a complicated mixture distribution (\ref{eq:mixture}) and enables understanding of the behavior of outlier components in high dimensions. The main contribution of our theorems compared to the existing theories lie in that data observations do not follow the same distribution, which mean the well-known theories must be extended. 

We consider increasing sample size $n$, increasing dimension $d$, and increasing spike signals.  As an indication of increasing spike signals, we let $\lam_i$, $\tauif$, and $\tauis$ be sequences indexed by $n$, that is, $\lamn_i$, $\tauifn$, and $\tauisn$. Consider the $M+1$ tiers where the first $K$ eigenvalues, $\{ \lamn_i \}_{1\leq i \leq K}$, are grouped such that $q_m$ eigenvalues fall into the $m$-th tier where $\sum_{m=1}^{M} q_m = K$ and the rest of the eigenvalues are all grouped into the $M+1$-th tier. Define $q_0 = 0, q_{M+1} = d - K$, and the partial sums $p_m = \sum_{l=0}^{m}q_l$. Then, the index set of the eigenvalues in the $m$-th tier can be written as
	\begin{eqnarray*}
		H_m = \big\{ p_{m-1} +1, p_{m-1} +2, \cdots, p_{m-1} + q_m \big\}~~~\mathrm{for}~~m=1, \cdots, M+1.
	\end{eqnarray*}
Denote a linear subspace spanned by the components in the $m$-th tier by $S_m = \mathrm{span} \{ U_i , i \in H_m \}$ for $m = 1, \cdots, M+1$.

The following assumptions provide the conditions for the variances, $\tauifn$ and $\tauisn$, of the underlying mixtures. Several different conditions are assumed for main spike signals, outlier spike signals, and noise signals, which helps to distinguish spike components from non-spike components. There are two types of noise in our model. One type is noise for all data points that correspond to the non-spike components in the model. By contrast, the other type is noise for the majority but a signal for a few observations. The latter type of noise is modeled by the small variance part in the outlier components. The following two assumptions illustrate the variances for these two types of noise.
\begin{assumption}\label{assume:nonspike}
$ \limunder{n \rightarrow \infty} \tauifn = \limunder{n \rightarrow \infty} \tauisn = c_{\lam}$~~ for $i \in H_{M+1}$.  
\end{assumption}
\begin{assumption}\label{assume:tau1_out}
$\limunder{n \rightarrow \infty} \tauifn = c_{\lam}$~~ for $i \in \iout$.
\end{assumption}
In a spike covariance model, noise signals are described in non-spike components and the corresponding underlying eigenvalues often are assumed to be constant for modeling white noise. This helps the bulk eigenvalues corresponding to the noise possess some known asymptotic properties. For instance, the distribution of the bulk eigenvalues converges to some well-known distributions, e.g. the Marcenko-Pastur law or the semi-circular law, and the extreme eigenvalues (the smallest and largest eigenvalues) are also known to be consistent to some values or asymptotically follow the Tracy-Widom distribution \citep{marvcenko1967distribution,bai1988convergence,bai1988necessary,bai1993,johnstone2001distribution}. In the same spirit, Assumption \ref{assume:nonspike} describes the asymptotically equivalent noise signals for non-spike directions $\{U_i\}_{i \in H_{M+1}}$. Eventually, the underlying eigenvalues $\{\lamn_i\}_{i > K}$ are all equal to $c_{\lam}$ for large $d$. Assumption \ref{assume:tau1_out} describes noise variances ($\tauifn$) for the outlier spike components. Since the outlier spike components are nothing but noise for the majority of the data, the same level of variation assumed for the non-spike components can be assumed. Thus, the noise variances for outlier spike directions are also asymptotically equal to $c_{\lam}$. This nicely connects the outlier model with the null model, i.e. the case with no outlier spike components, in the sense that the outlier components will merge with non-spike noise components. 

In contrast to noise signals, we allow spike signals to be increasing in $n$. The intensity of each spike component is determined by the underlying variation that each component is involved in, which is equivalent to its corresponding eigenvalue. For large $n$, the underlying eigenvalues, $\lamn_i$, are simply $\tauifn$ for $i \in \imain$ whereas, for $i \in \iout$, the eigenvalues are $w_i \tauisn$ because variation from the larger variance component $\tauisn$ dominate variation from the smaller variance component $\tauifn$. The PCA consistency strongly depends on the magnitudes of spike eigenvalues, which are specified in a systematic manner in the following assumptions. Let $ \delmn $ for $m=1, \cdots, M$ be sequences of constant values for index $n$. 
\begin{assumption}\label{assume:spike}
$\limunder{n \rightarrow \infty} \frac{\tauifn }{\delmn} = 1$ for $i \in H_m \cap \imain $ and $\limunder{n \rightarrow \infty} \frac{\tauisn}{w_i\delmn} = 1$ for $i \in H_m \cap \iout $, $m= 1, \cdots, M$.
\end{assumption}
\begin{assumption}\label{assume:tier_rates}
As $n \rightarrow \infty$, $\delta_1^{(n)} \succ \delta_2^{(n)} \succ \cdots \succ \delta_M^{(n)} \succ \lamn_{K+1}$ where $a_n \succ b_n$ implies $\limunder{n \rightarrow \infty} \frac{a_n}{b_n}>1$.
\end{assumption}
Assumption \ref{assume:spike} allows the components in the same tier to share asymptotically equivalent eigenvalues. We further assume different limiting coefficients for different tiers in Assumption \ref{assume:tier_rates}, which enables the characterization of the $M$ subspaces spanned by the directions in each tier.

Under Assumptions \ref{assume:nonspike}-\ref{assume:tier_rates}, we now investigate the asymptotic properties of the sample eigenvalues and eigenvectors. Even though we assume a complicated mixture distribution for the underlying structure and thus the data are not i.i.d., we obtain parallel asymptotic results to those in \cite{shen2016general}. This is because although observations are not from an identical distribution, they are allowed to share the same underlying eigenvectors from the model (\ref{eq:mixture}). Then, we have a simple integrated covariance matrix $\Sig$ so that a spike covariance model can be employed even when data come from multiple distributions. 

In general, the strength of underlying spike signals and increasing sample size $n$ encourage PCA consistency whereas increasing dimension $d$ discourages consistency. When the underlying spike signals in the $m$-th tier with increasing $n$ are asymptotically strong enough to prevail over the dimension $d$ in the sense that $\frac{d}{n\deln_{m}} \rightarrow 0 $, it follows that the estimates of the eigenvectors are subspace consistent in the $m$-th tier and the estimates of the eigenvalues are consistent as well. Theorems \ref{thm1:eigenvalue} and \ref{thm2:eigenvector} demonstrate such asymptotic behavior in a concrete manner under different scenarios. 

\begin{theorem}\label{thm1:eigenvalue}
Under Assumptions \ref{assume:nonspike}-\ref{assume:tier_rates}, 
\begin{itemize}
    \item[(a)] if $\frac{d}{n\delta_{M}^{(n)}} \rightarrow 0 $, then
    \begin{itemize}
        \item[(i)] for $i \leq K$, $\frac{\hat{\lam}_i}{\lamn_i} \rightarrow_{a.s.} 1$ where $\lamn_i=\tauifn$ for $i \in \imain$ and $\lamn_i=w_i\tauisn$ for $i \in \iout $; 
        \item[(ii)] for $i>K$, 
        \begin{itemize}
            \item[$\bullet$] if $0<c<\infty$, $c_{\lam}(1-\sqrt{c})^2 \leq \hatlam_{n \wedge d} \leq \hatlam_1 \leq c_{\lam}(1+\sqrt{c})^2$ a.s.;
            \item[$\bullet$] if $c=\infty$, $\frac{n\hatlam_i}{d} \rightarrow_{a.s.} c_{\lam}$;
            \item[$\bullet$] if $c=0$, $\hatlam_i \rightarrow_{a.s.} c_{\lam}$;
        \end{itemize}
    \end{itemize}
    \item[(b)] if $\frac{d}{n\delta_{h}^{(n)}} \rightarrow 0 $ where $1 \leq h < M$ and $\frac{d}{n\delta_{h+1}^{(n)}} \rightarrow \infty $, then
    \begin{itemize}
        \item[(i)]  for $i  \leq p_h$, $\frac{\hat{\lam}_i}{\lamn_i} \rightarrow_{a.s.} 1$ where $\lamn_i=\tauifn$ for $i \in \imain$ and $\lamn_i =w_i\tauisn$ for $i \in \iout $;
        \item[(ii)] for $i>p_h$, $\frac{n\hatlam_i}{d} \rightarrow_{a.s.} c_{\lam}$.
    \end{itemize}
\end{itemize}
\end{theorem}
Theorem \ref{thm1:eigenvalue} considers two scenarios: (a) when all spike signals are strong and (b) when strong population signals are assumed only up to the $h$-th tier and the other signals are dominated by the increasing dimension, i.e. $\frac{d}{n\deln_{h+1}} \rightarrow \infty$. It should be noted that these two different scenarios yield different asymptotic regimes: (a) considers all three cases for the limit of $c$, i.e. $0<c<\infty$, $c=\infty$, and $c=0$, whereas Theorem \ref{thm1:eigenvalue} (b) considers only the case of $c = \infty$. This is because the condition $\frac{d}{n\deln_{h+1}} \rightarrow \infty$ of (b) and Assumption \ref{assume:tier_rates} together rule out the cases of $c < \infty$ as $\frac{d}{n\lamn_{K+1}} \rightarrow \infty$ can hold only when $\frac{d}{n} \rightarrow \infty$. In both cases, if the signal in a tier is strong enough so that $\frac{d}{n\deln} \rightarrow 0$, then the sample eigenvalues corresponding to the tier consistently estimate the true eigenvalues. On the other hand, if the spike signals are not that strong, then the corresponding sample eigenvalues tend to be swallowed by the small bulk eigenvalues. 

Intuitively, although an underlying outlier component is dramatically intense, its realized signal is much weaker than the true one because it loses the power due to the small chance of participation. Assumption \ref{assume:spike} reflects this intuition and gives a condition that the $i$th outlier spike signal should be $1/w_i$ times greater than the other main spike signals in the same tier to compensate for this loss of power. Based on this assumption, Theorem \ref{thm1:eigenvalue} demonstrates that such an outlier signal would asymptotically attain the same sample eigenvalues as the main signals in the same tier. In particular, the sample eigenvalue from an outlier signal converges to the dominating variance ($\tauisn$) multiplied by the corresponding proportion ($w_i$) in the underlying mixture distribution (\ref{eq:mixture}). Therefore, the true levels of outlier signals can be approximately estimated by dividing the corresponding eigenvalues by the proportion ($\approx w_i$) of the relevant outliers. 

In many outlier detection methods, it is of great interest to choose the subspace that outlier components are involved in \citep{filzmoser2008outlier,ahn2018distance}. Although Theorem \ref{thm1:eigenvalue} suggests that a few large sample eigenvalues may consistently estimate the true levels of the signals, it is not enough to say that the corresponding principal component directions construct a useful subspace for detecting outliers. This brings to the study of eigenvectors that is discussed in the following theorem. Let $\deln_0 = \infty$ for all $n$.


\begin{theorem}\label{thm2:eigenvector}
Under Assumptions \ref{assume:nonspike}-\ref{assume:tier_rates}, 
\begin{itemize}
    \item[(a)] if $\frac{d}{n\delta_{M}^{(n)}} \rightarrow 0 $, and $0 < c \leq \infty$, then
    \begin{itemize}
        \item[(i)] $\hat{U}_i$ are subspace consistent in the sense that the $\mathrm{angle}(\hat{U}_i, S_m) \rightarrow_{a.s.} 0$ for $i \in H_m $, $ m = 1, \cdots, M+1$.
        \begin{itemize}
		      \item[$\bullet$] For $m=1, \cdots, M-1$, $\mathrm{angle}(\hat{U}_i, S_m) = o(\big\{ \frac{\delmn}{\deln_{m-1}} \vee \frac{\deln_{m+1}}{\deln_{m}} \big\}^{1/2})$.
		      \item[$\bullet$] For $m=M$, $\mathrm{angle}(\hat{U}_i, S_m) = o(\big\{ \frac{\deln_{m}}{\deln_{m-1}} \big\}^{1/2}) \vee O(\big\{ \frac{d}{n\deln_{m}} \big\}^{1/2})$.
		      \item[$\bullet$] For $m=M+1$, $\mathrm{angle}(\hat{U}_i, S_m) = O(\big\{ \frac{d}{n\deln_{m-1}} \big\}^{1/2}) $.
	      \end{itemize}
    \end{itemize}    
    \item[(b)] if $\frac{d}{n\delta_{h}^{(n)}} \rightarrow 0 $ where $1 \leq h < M$ and $\frac{d}{n\delta_{h+1}^{(n)}} \rightarrow \infty $, then
    	\begin{itemize}
    	   \item[(i)] $\hat{U}_i$ for $i \leq p_h $ are subspace consistent in the sense that the $\mathrm{angle}(\hat{U}_i, S_m) \rightarrow_{a.s.} 0$ for $i \in H_m $, $ m = 1, \cdots, h$.
    	   \begin{itemize}
    	       \item[$\bullet$] For $m=1, \cdots, h-1$, $\mathrm{angle}(\hat{U}_i, S_m) = o(\big\{ \frac{\deln_{m}}{\deln_{m-1}} \vee \frac{\deln_{m+1}}{\deln_{m}} \big\}^{1/2})$.
    	       \item[$\bullet$] For $m=h$, $\mathrm{angle}(\hat{U}_i, S_m) = o(\big\{ \frac{\deln_{m}}{\deln_{m-1}} \big\}^{1/2}) \vee O(\big\{ \frac{d}{n\deln_{m}} \big\}^{1/2})$.
    	   \end{itemize}
    	   \item[(ii)] $\hat{U}_i$ for $i > p_h$ are strongly inconsistent in the sense that $|<\hat{U}_i, U_i>| = O(\big\{ \frac{n\lamn_i}{d} \big\}^{1/2})$.
    	\end{itemize} 
\end{itemize}
\end{theorem}
Under the same scenarios considered in Theorem \ref{thm1:eigenvalue}, Theorem \ref{thm2:eigenvector} studies the asymptotic behavior of the sample eigenvectors in terms of angles as studied in \cite{jung2009pca} and \cite{shen2016general}. In each scenario, if the $m$th tier involves a strong signal such that $\frac{d}{n\deln_m}\rightarrow 0$, then the $i$th sample eigenvector, $\hat{U}_i$ for $i \in H_m$, tends to be in the subspace, $S_m$, which is spanned by the underlying directions in the $m$th tier. This holds for all $i \in H_m$, and thus it follows that the subspace spanned by the sample eigenvectors, $\{\hat{U}_i \}_{i \in H_m}$, converges to the $S_m$. This phenomenon is called the PCA subspace consistency. Also, the different levels of signals in different tiers assumed in Assumption \ref{assume:tier_rates} make the estimated subspaces become distinct for large $n$ and large $d$, and more gaps between the levels accelerate this distinction as indicated by the different convergence rates obtained in the theorem.

In high-dimensional data, searching for outliers in a much lower dimensional subspace where the outliers are distinguishable is advantageous. The subspace often provides critical information for interpreting why an object is outlying and to what extent the object is an outlier \citep{kamber2001data}. This is almost impossible using full dimensions because of the overwhelming noise. Once such a subspace is found, some appropriate conventional outlier detection methods may be applicable for the approximated low-dimensional data. Under the assumption that all outlier signals dominate the dimensions, combining Theorem \ref{thm2:eigenvector} (a) with Theorem \ref{thm1:eigenvalue} (a) allows one to find the outlier-relevant low dimensional subspace by using the first few PC directions whose sample eigenvalues are substantially large and thus separate from the other bulk eigenvalues. 

On the other hand, Theorem \ref{thm2:eigenvector} (b) together with Theorem \ref{thm1:eigenvalue} (b) shows that when only a subset of the outlier signals are strong enough to dominate the increasing dimensions, the first few PC directions with large sample eigenvalues provide a good subspace only for those strong outlier components. Not only this, the strong inconsistency suggests that it becomes very challenging to distinguish the outlier directions missing from the first few PCs from the non-spike directions. This is not simply because the sample eigenvalues from those weak outlier signals are not separable from the bulk sample eigenvalues. Once the spike samples eigenvalues are swallowed by the bulk, then it is likely that the corresponding directions are all mixed with non-spike directions, so any of the single sample eigenvectors may not be representative of those spike directions. Therefore, the approximation of the data matrix using the first few eigenvectors may miss some important information that are relatively weak but not noise. This implies that the outlier components with weak signals or with extremely small participation are harder to be separated from noise. Thus special care should be taken to find the hidden outlying structure.

Also, it should be noted that the theorem does not guarantee that the sample eigenvectors are individually consistent to the true ones. So looking at the individual PC directions may not be enough to detect outliers. To illustrate this situation, a toy example is given in Section \ref{ch_theory:toy_example}. As one of the special and important cases, we now consider the case when all spike eigenvalues are separable, i.e. $q_1 = q_2 = \cdots = q_M = 1$ and $M=K$. Then, Assumption \ref{assume:tier_rates} becomes 
\begin{assumption}\label{assume:individual_rates}
	As $n \rightarrow \infty$, $\lamn_1 \succ \lamn_2 \succ \cdots \succ \lamn_K \succ \lamn_{K+1} > 0$.
\end{assumption}
This allows us to get the individual consistency of eigenvalues as well as eigenvectors instead of subspace consistency. The following corollaries of Theorem \ref{thm1:eigenvalue} and \ref{thm2:eigenvector} describe such individual consistency under the same scenarios with the respective theorems. 

\begin{corollary}\label{cor1:eigenvalue}
	Under Assumptions \ref{assume:nonspike}, \ref{assume:tau1_out}, and \ref{assume:individual_rates}, 
	\begin{itemize}
    \item[(a)] if $\frac{d}{n\lamn_{K}} \rightarrow 0 $, then
    \begin{itemize}
        \item[(i)] for $i \leq K$, $\frac{\hat{\lam}_i}{\lamn_i} \rightarrow_{a.s.} 1$ where $\lamn_i=\tauifn$ for $i \in \imain$ and $\lamn_i=w_i\tauisn$ for $i \in \iout $; 
        \item[(ii)] for $i>K$, 
        \begin{itemize}
            \item[$\bullet$] if $0<c<\infty$, $c_{\lam}(1-\sqrt{c})^2 \leq \hatlam_{n \wedge d} \leq \hatlam_1 \leq c_{\lam}(1+\sqrt{c})^2$ a.s.;
            \item[$\bullet$] if $c=\infty$, $\frac{n\hatlam_i}{d} \rightarrow_{a.s.} c_{\lam}$;
            \item[$\bullet$] if $c=0$, $\hatlam_i \rightarrow_{a.s.} c_{\lam}$;
        \end{itemize}
    \end{itemize}
    \item[(b)] if $\frac{d}{n\lamn_{h}} \rightarrow 0 $ where $1 \leq h <K$ and $\frac{d}{n\lamn_{h+1}} \rightarrow \infty $, then
    \begin{itemize}
        \item[(i)]  for $i  \leq h$, $\frac{\hat{\lam}_i}{\lamn_i} \rightarrow_{a.s.} 1$ where $\lamn_i=\tauifn$ for $i \in \imain$ and $\lamn_i =w_i\tauisn$ for $i \in \iout $;
        \item[(ii)] for $i>h$, $\frac{n\hatlam_i}{d} \rightarrow_{a.s.} c_{\lam}$.
    \end{itemize}
\end{itemize}
\end{corollary}

\begin{corollary}\label{cor2:eigenvector}
	Under Assumptions \ref{assume:nonspike}, \ref{assume:tau1_out}, and \ref{assume:individual_rates}, 
\begin{itemize}
    \item[(a)] if $\frac{d}{n\lamn_{K}} \rightarrow 0 $, and $0 < c \leq \infty$, then
    \begin{itemize}
        \item[(i)] $\hat{U}_i$ are consistent with $U_i$ in the sense that $\mathrm{angle}(\hat{U}_i, U_i) \rightarrow_{a.s.} 0$ for $i=1,\cdots,K$.
        \begin{itemize}
		      \item[$\bullet$] For $i=1, \cdots, K-1$, $\mathrm{angle}(\hat{U}_i, U_i) = o(\big\{ \frac{\lamn_i}{\lamn_{i-1}} \vee \frac{\lamn_{i+1}}{\lamn_{i}} \big\}^{1/2})$.
		      \item[$\bullet$] For $i=K$, $\mathrm{angle}(\hat{U}_i, U_i) = o(\big\{ \frac{\lamn_{i}}{\lamn_{i-1}} \big\}^{1/2}) \vee O(\big\{ \frac{d}{n\lamn_{i}} \big\}^{1/2})$.
		      \item[$\bullet$] For $i>K$, $\mathrm{angle}(\hat{U}_i, S) = O(\big\{ \frac{d}{n\lamn_{i-1}} \big\}^{1/2}) $ where $S=\mathrm{span}(U_i : i>K)$.
	      \end{itemize}
    \end{itemize}    
    \item[(b)] if $\frac{d}{n\lamn_{h}} \rightarrow 0 $ where $1 \leq h < K$ and $\frac{d}{n\lamn_{h+1}} \rightarrow \infty $, then
    	\begin{itemize}
    	   \item[(i)] $\hat{U}_i$ are consistent with $U_i$ in the sense that $\mathrm{angle}(\hat{U}_i, U_i) \rightarrow_{a.s.} 0$ for $i=1,\cdots,h$.
    	   \begin{itemize}
    	       \item[$\bullet$] For $i=1, \cdots, h-1$, $\mathrm{angle}(\hat{U}_i, U_i) = o(\big\{ \frac{\lamn_{i}}{\lamn_{i-1}} \vee \frac{\lamn_{i+1}}{\lamn_{m}} \big\}^{1/2})$.
    	       \item[$\bullet$] For $i=h$, $\mathrm{angle}(\hat{U}_i, U_i) = o(\big\{ \frac{\lamn_{i}}{\lamn_{i-1}} \big\}^{1/2}) \vee O(\big\{ \frac{d}{n\lamn_{i}} \big\}^{1/2})$.
    	   \end{itemize}
    	   \item[(ii)] $\hat{U}_i$ for $i > h$ are strongly inconsistent in the sense that $|<\hat{U}_i, U_i>| = O(\big\{ \frac{n\lamn_i}{d} \big\}^{1/2})$.
    	\end{itemize} 
\end{itemize}
\end{corollary}

\section{Illustration using a toy example}\label{ch_theory:toy_example}
\noindent We now illustrate the PCA subspace consistency with a toy example under the model (\ref{eq:mixture}), highlighting the situation where an outlier component is captured by the first few PC directions but none of the PC directions are individually representative of the outlier component. 

First, let us describe the simulation setting. We generated $n=200$ independent data vectors in $d=3000$ dimensions based on our model described in (\ref{eq:mixture}). To generate such data, $\{z_{ij} \}_{1 \leq i \leq d, 1 \leq j \leq n}$ are assumed to be distributed as independent $N(0,1)$ and the standard basis vectors, $\{e_i \}_{1 \leq i \leq d}$, are used as underlying eigenvectors $\{U_i \}_{1 \leq i \leq d}$ with $e_1, \cdots, e_9$ being the main spike directions and $e_{10}$ being an outlier spike direction. For the main spike directions, the underlying variations are assumed to be $\tau_{i,1}=3000,1000,100,90,80,70,60,50,40$ for $i=1, \cdots, 9$. For the outlier spike direction $e_{10}$, we assume $\tau_{10,1}=2000, \tau_{10,2}=1$ and the outlier proportion $w_{10}=0.02$. For the other non-spike directions $\{ U_i \}_{11 \leq i \leq d}$ corresponding to noise, $\tau_{i,1}=1$ and $\tau_{i,2}=1$ are assumed. A realization from this model had the $4$ outliers, denoted by $X_{1}, \cdots, X_{4}$, with $196$ normal data points, denoted by $X_{5}, \cdots, X_{200}$.

For this data set, we constructed a sample covariance matrix where PCA was applied and obtained a set of sample eigenvectors and eigenvalues. Since the true spike directions are $e_1, \cdots, e_{10}$, we can examine the contribution of each sample eigenvector onto the true spike directions simply by taking the squares of the entries. The sum of the squares of entries in each sample eigenvector is one and thus the squared values $\hat{u}_{ji}^2$, i.e. the squared $j$th entry of $\hatU_i$, can be regarded as the explained percentage of the underlying vector $e_j$ in the direction $\hatU_i$.  Table \ref{tab:toy_example_eigenvectors} gives the squares of the first 12 entries (in rows) of the first 11 eigenvectors $\hatU_1, \cdots, \hatU_{11}$ (in columns). The last two rows indicate the corresponding sample eigenvalues $\hatlam_i$ and the angles between the true outlier direction $e_{10}$ and $\hatU_{i}$ for $i=1, \cdots, 11$. The largest value in each $\hatU_i$ is indicated using red, and if the red value, say $\hat{u}_{ji}^2$, is close to one and the other entries are close to zero, then the $\hatU_i$ is a good estimate of the $e_j$. For example, the first entry of $\hatU_1$ is approximately one with all the other entries of zero, indicating that the first underlying direction $e_1$ is well estimated by $\hatU_1$. Similarly, $\hatU_2$ is a good estimate of the $e_2$.

On the other hand, none of the $\hatU_3, \cdots, \hatU_{10}$ has an entry which is close to one. Instead, they have several nonzero entries, indicating that each of them has some correlation with several underlying directions. This can be understood that any of the underlying directions, $e_3, \cdots, e_{10}$ are well estimated by the single sample eigenvectors. Nonetheless, an important note is that, for each row $j=3, \cdots, 10$, the sum of the squared $j$th entries of $\hatU_3, \cdots, \hatU_{10}$ is close to one. This supports the PCA subspace consistency in that each of the true eigenvectors $e_3, \cdots, e_{10}$ can be estimated by a linear combination of $\hatU_3, \cdots, \hatU_{10}$ rather than any individual directions. As described in Theorem \ref{thm2:eigenvector}, this is because the underlying variation in $e_3, \cdots, e_{10}$ are nearly in the same tier, which tends to somewhat discourage the individual consistency. 

In particular, it should be noted that none of the first 10 eigenvectors alone provide good estimates for the outlier direction $e_{10}$ as highlighted in lightblue. Specifically, there is no direction that describes the $e_{10}$ more than $30\%$. The angles in the last row also reveal that none of those 10 sample eigenvectors are close to $e_{10}$. However, the sum of the squared 10th entries in the first 10 PC directions, $\sum_{i=1}^{10}\hat{u}_{i,10}^2$, is almost $0.88$, indicating that the $e_{10}$ may be well captured by the subspace spanned by the 10 sample eigenvectors. As discussed earlier, therefore, this supports the concept that although using individual PC directions for detecting outliers may be ineffective, this subspace does preserve the critical information for the outliers and thus may be used to detect those outliers. 


\begin{landscape}
\begin{table}
\centering
\footnotesize
\begin{tabular}{rrrrrrrrrrrr}
  \hline
 & $\hat{U}_1$ & $\hat{U}_2$ & $\hat{U}_3$ & $\hat{U}_4$ & $\hat{U}_5$ & $\hat{U}_6$ & $\hat{U}_7$ & $\hat{U}_8$ & $\hat{U}_9$ & $\hat{U}_{10}$ & $\hat{U}_{11}$ \\ 
  \hline
1 & \textcolor{red}{0.996} & 0.000 & 0.000 & 0.000 & 0.000 & 0.000 & 0.000 & 0.000 & 0.000 & 0.000 & 0.000 \\ 
  2 & 0.000 & \textcolor{red}{0.987} & 0.000 & 0.000 & 0.000 & 0.000 & 0.000 & 0.000 & 0.001 & 0.000 & 0.000 \\ 
  3 & 0.000 & 0.000 & \textcolor{red}{0.657} & 0.001 & 0.072 & 0.173 & 0.004 & 0.000 & 0.000 & 0.002 & 0.000 \\ 
  4 & 0.000 & 0.000 & 0.020 & \textcolor{red}{0.665} & 0.111 & 0.023 & 0.055 & 0.011 & 0.001 & 0.003 & 0.000 \\ 
  5 & 0.000 & 0.001 & 0.029 & 0.031 & \textcolor{red}{0.330} & \textcolor{red}{0.422} & 0.048 & 0.008 & 0.001 & 0.013 & 0.000 \\ 
  6 & 0.000 & 0.000 & 0.071 & 0.006 & 0.198 & 0.023 & 0.064 & \textcolor{red}{0.501} & 0.007 & 0.003 & 0.000 \\ 
  7 & 0.000 & 0.000 & 0.036 & 0.002 & 0.080 & 0.009 & \textcolor{red}{0.408} & 0.201 & 0.053 & 0.059 & 0.000 \\ 
  8 & 0.000 & 0.000 & 0.000 & 0.007 & 0.004 & 0.021 & 0.094 & 0.047 & 0.243 & \textcolor{red}{0.382} & 0.000 \\ 
  9 & 0.000 & 0.001 & 0.003 & 0.000 & 0.012 & 0.001 & 0.010 & 0.000 & \textcolor{red}{0.467} & 0.292 & 0.000 \\ 
  \rowcolor{LightCyan}
  10 & 0.000 & 0.000 & 0.101 & 0.182 & 0.079 & 0.215 & 0.169 & 0.096 & 0.030 & 0.007 & 0.000 \\ 
  11 & 0.000 & 0.000 & 0.000 & 0.000 & 0.000 & 0.000 & 0.000 & 0.000 & 0.000 & 0.000 & 0.000 \\ 
  12 & 0.000 & 0.000 & 0.000 & 0.000 & 0.000 & 0.000 & 0.000 & 0.000 & 0.000 & 0.000 & 0.000 \\ 
       & \vdots & \vdots & \vdots & \vdots & \vdots & \vdots & \vdots & \vdots & \vdots & \vdots & \vdots \\
   \hline
  $\hatlam_i$ & 3519.209 & 996.408 & 123.856 & 99.055 & 91.825 & 86.694 & 71.458 & 70.393 & 52.236 & 42.087 & 16.850 \\ 
  \hline
    \rowcolor{LightCyan}
  angle & 89.6 & 89.1 & 71.5 & 64.8 & 73.7 & 62.4 & 65.7 & 71.9 & 80.0 & 85.3 & 89.4 \\   
  \hline
\end{tabular}
\caption[Sample eigenvectors]{\small{This table shows the first 11 sample eigenvectors with squared entries. The largest value in each $\hatU_i$ is colored using a red font and the row corresponding to the outlier signal, $e_{10}$, is hightlighted using lightblue background. This row indicates how much each $\hatU_i$ explains the outlier signal. The last two rows respectively show the sample eigenvalues corresponding to $\{\hatU_i\}_{1 \leq i \leq 11}$ and the angles between $\hatU_i$ and the true outlier signal, $e_{10}$. }}
\label{tab:toy_example_eigenvectors}
\end{table}
\end{landscape}	

\section{Proofs}\label{ch_theory:proofs}
\noindent In this section, we provide proofs for the theorems in Section \ref{ch_theory:asymptotics}. The main steps in the proofs are similar to the proofs in \cite{shen2016general} but our different setting for the underlying distribution of data requires the addition of more detail.

Let $\Y$ be a $d \times n$ matrix whose column vectors are $ Y_1, Y_2, \cdots, Y_n$ where $Y_j = (y_{1j},\cdots, y_{dj})^{T}$ and $y_{ij}$'s are independent random variables in our model in (\ref{eq:mixture}). Then, switching the roles of columns and rows, we get the $n \times n$ dual (Gram) matrix of the sample covariance matrix $\hat{\Sig}$
\begin{eqnarray*}
\dual{\samcov}=\frac{1}{n}\X^{T}\X=\frac{1}{n}\Y^{T}\U^{T} \U \Y = \frac{1}{n}\Y^{T}\Y,
\end{eqnarray*}
and it is well known that they share the same nonzero eigenvalues. Let us define two matrices that will be treated separately in the proof. Let 
\begin{eqnarray*}
	\A = \frac{1}{n} \sum_{i=1}^{K}\tilY_i \tilY_i^{T}, ~~~\mathrm{and}~~\B= \frac{1}{n} \sum_{i=K+1}^{d}\tilY_i \tilY_i^{T},
\end{eqnarray*}
where $\tilY_i$ is the $i$-th row vector of $\Y$.
Then, 
\begin{eqnarray*}
	\dual{\samcov} & = & \frac{1}{n} \sum_{i=1}^{d}\tilY_i \tilY_i^{T} 
	 =  \A + \B.
\end{eqnarray*}
\indent Before the proof, we provide two popular lemmas. Lemma \ref{lem:Weyl_ineq} provides the upper and lower bounds for the eigenvalues of a matrix that can be expressed as the sum of two symmetric matrices.  

\begin{lemma}\label{lem:Weyl_ineq}
(Weyl inequality) Let $A$ and $B$ be $n \times n$ real symmetric matrices. Then, for all $j,k,l = 1, \cdots, n$, 
\begin{eqnarray*}
	\lam_k(A)+\lam_l(B) & \leq & \lam_j(A+B)~~~\mathrm{for}~k+l=j+n \\
	\lam_k(A)+\lam_l(B) & \geq & \lam_j(A+B)~~~\mathrm{for}~k+l=j+1
\end{eqnarray*}	
where $\lam_j(A)$ is the $j$-th largest eigenvalue of a matrix $A$.
\end{lemma}

Next, Lemma \ref{lem:Bai-Yin} provides the convergence of the largest and smallest non-zero eigenvalues of a random matrix, which is known as Bai-Yin's law \citep{bai1993}. 
\begin{lemma}\label{lem:Bai-Yin} (Bai-Yin's law)
	Suppose $B=\frac{1}{q}VV^{T}$ where $V$ is a $p \times q$ random matrix composed of i.i.d. random variables with zero mean, unit variance and finite fourth moment. As $q \rightarrow \infty$ and $\frac{p}{q} \rightarrow c \in [0, \infty)$, the largest and smallest non-zero eigenvalues of $B$ converge almost surely to $(1+\sqrt{c})^2$ and $(1-\sqrt{c})^2$, respectively.  
\end{lemma}

\subsection{Proof of Theorem \ref{thm1:eigenvalue}}
\begin{proof}
\indent The proof consists of the following three steps:
\begin{enumerate}
	\item Establish the convergence of $\eigval{k}{\A}$.
	\item Establish the convergence of $\eigval{k}{\B}$.
	\item Establish the convergence of $\eigval{k}{\A + \B}$.
\end{enumerate} 

Lemma \ref{lem:conv_A} proves the first step.
\begin{lemma}\label{lem:conv_A}
	As $n \rightarrow \infty$, we have
\begin{eqnarray*}
	\frac{1}{\lamn_{k}}\eigval{k}{\A} \rightarrow 1~~\mathrm{a.s.}~~~\mathrm{for}~~k=1, \cdots, K.
\end{eqnarray*}
where $\lamn_i = \tauifn$ for $i \in \imain$ and $\lamn_i = w_i \tauisn$ for $i \in \iout$.
\end{lemma}
\begin{proof}
	Define $\A_{k}=\frac{1}{n}\sum_{i=k}^{K}\tilY_i \tilY_i^{T}$ with $\A_{k,D}$ being its dual matrix and $\A_{k,R}=\A-\A_{k}$. Then, 
\begin{eqnarray}\label{eq:dual_AB}
	\eigval{1}{\frac{1}{n}\tilY_k \tilY_k^{T}} + \eigval{n}{\frac{1}{n}\sum_{i=k+1}^{K}\tilY_i \tilY_i^T} \leq \eigval{k}{\A} \leq \eigval{1}{\A_k} + \eigval{k}{\A_{k,R}}
\end{eqnarray}
where the upper bound follows from Lemma \ref{lem:Weyl_ineq} and the lower bound from the expression (5.9) in \cite{jung2009pca}. Since the rank of $\frac{1}{n}\sum_{i=k+1}^{K}\tilY_i \tilY_i^T$ is less than $K<n$, $\eigval{n}{\frac{1}{n}\sum_{i=k+1}^{K}\tilY_i \tilY_i^T} $ should be zero. Likewise, the rank of $\A_{k,R}$ is at most $k-1$, and thus $\eigval{k}{\A_{k,R}}=0$. Therefore, we get
\begin{eqnarray}\label{eq:ineq_A}
	\eigval{1}{\frac{1}{n}\tilY_k \tilY_k^{T}} \leq \eigval{k}{\A} \leq \eigval{1}{\A_k}.
\end{eqnarray} 
By dividing (\ref{eq:ineq_A}) by $\lamn_k$, the inequality becomes 
\begin{eqnarray}\label{eq:ineq_A2}
		\frac{1}{\lamn_k}\eigval{1}{\frac{1}{n}\tilY_k \tilY_k^{T}} \leq \frac{1}{\lamn_k}\eigval{k}{\A} \leq \frac{1}{\lamn_k}\eigval{1}{\A_k}.
\end{eqnarray} 
We now show that the left hand side converges to $1$. Note that $\eigval{1}{\frac{1}{n}\tilY_k \tilY_k^{T}}=\eigval{1}{\frac{1}{n}\tilY_k^{T} \tilY_k}$ and thus we show the convergence of $\frac{1}{\lamn_k}\frac{1}{n}\tilY_k^{T}\tilY_k$. For $k=1,\cdots, d$, let $s_k = \{1 \leq j \leq n: y_{kj} = \sqrt{\taun{k,2}}z_{kj}\}$, which is an index set containing sample indices from the second component in the mixture model (\ref{eq:mixture}) corresponding to the direction $U_k$. Then,
\begin{eqnarray}\label{eq:diagonal}
	\frac{1}{n}\tilY_k^{T}\tilY_k & = & \frac{1}{n}\sum_{j=1}^{n} y_{kj}^2 \nonumber \\
	& = & \frac{1}{n}\sum_{j \in s_k^{c}} \taun{k,1}z_{kj}^2 + \frac{1}{n}\sum_{j \in s_k} \taun{k,2}z_{kj}^2  \nonumber \\
	& = & \taun{k,1}\frac{|s_k^{c}|}{n}\frac{1}{|s_k^{c}|}\sum_{j \in s_k^{c}} z_{kj}^2 + \taun{k,2}\frac{|s_k|}{n}\frac{1}{|s_k|}\sum_{j \in s_k} z_{kj}^2
\end{eqnarray}
where $|s|$ is the cardinality of a set $s$. If $k \in \imain$, then $s_k = \emptyset$ and $\lamn_k = \taun{k,1}$, and thus it follows from the law of large number that $\frac{1}{\lamn_k}\frac{1}{n}\tilY_k^{T}\tilY_k \rightarrow 1$ almost surely. If $k \in \iout$, then we have $\lamn_k = w_k \taun{k,2}$, $\frac{|s_k^{c}|}{n} \rightarrow 1-w_k$, $\frac{|s_k|}{n} \rightarrow w_k$. Also, since $\frac{\taun{k,1}}{\taun{k,2}}\rightarrow 0$ as $n \rightarrow \infty$, the convergence $\frac{1}{\lamn_k}\frac{1}{n}\tilY_k^{T}\tilY_k \rightarrow 1$ for $k \in \iout$ also follows from the law of large numbers. Hence we conclude that 
\begin{eqnarray}\label{eq:leftconv}
	\frac{1}{\lamn_k}\eigval{1}{\frac{1}{n}\tilY_k \tilY_k^{T}} \rightarrow 1~~\mathrm{a.s.}
\end{eqnarray}
for $k=1, \cdots, K$.

Next, we show that the right hand side of (\ref{eq:ineq_A2}) also converges to $1$. Let $\A_{k,D}$ be the dual matrix of $\A_k$. Then, it can be written as
\begin{eqnarray*}
	\A_{k,D} = \frac{1}{n} \left(  \begin{array}{c c c c}
\tilY_k^{T}\tilY_k & \tilY_k^{T}\tilY_{k+1} & \cdots & \tilY_k^{T}\tilY_{K} \\
              &            &                              &                 \\
\vdots   &\ddots  &                              &   \vdots                 \\
                 &                              &             &                 \\
  \tilY_K^{T}\tilY_k     &           \cdots         &            &  \tilY_K^{T}\tilY_K
\end{array} \right). 
\end{eqnarray*}
Then, one can show that
\begin{eqnarray*}
	\frac{1}{\lamn_k}\A_{k,D} \longrightarrow \left( \begin{array}{c c c c}
 1 & 0 & \cdots & 0 \\
 0 & b_{k+1,k} & \cdots & 0 \\
 \vdots & 	& \ddots & \vdots \\
 0 & \cdots &    & b_{K,k}
 \end{array} \right)
\end{eqnarray*}
where $b_{l,k}=\lim_{n\rightarrow \infty}\frac{\lamn_{l}}{\lamn_{k}}$. The above convergence of each element of $\A_{k,D}$ can be obtained in a similar way as in (\ref{eq:diagonal}) and thus is omitted. It follows from $b_{l,k} \leq 1$ for $l \geq k$ that the largest eigenvalue of $\frac{1}{\lamn_k}\A_{k,D}$ converges almost surely $1$. Note that $\frac{1}{\lamn_k}\eigval{1}{\A_{k}} = \frac{1}{\lamn_k}\eigval{1}{\A_{k,D}}= \eigval{1}{\frac{1}{\lamn_k}\A_{k,D}} $ and thus it follows that
\begin{eqnarray}\label{eq:rightconv}
	\frac{1}{\lamn_{k}}\eigval{1}{\A_k} \rightarrow 1~~\mathrm{a.s.}
\end{eqnarray}
By (\ref{eq:leftconv}) and (\ref{eq:rightconv}), we conclude that
\begin{eqnarray*}
	\frac{1}{\lamn_{k}}\eigval{k}{\A} \rightarrow 1~~\mathrm{a.s.}
\end{eqnarray*}
\end{proof}

\begin{lemma}\label{lem:conv_B}
	As $n \rightarrow \infty$, we have 
	\begin{eqnarray}
	\eigval{\max}{\B} \rightarrow c_{\lam}(1+\sqrt{c})^2 ~ \mathrm{and} ~ \eigval{\min}{\B} \rightarrow c_{\lam}(1-\sqrt{c})^2~~\mathrm{a.s.} & \mathrm{if} & \frac{d}{n} \rightarrow c < \infty \nonumber \\
	\frac{n}{d}\eigval{\max}{\B} ~ \mathrm{and}~ \frac{n}{d}\eigval{\min}{\B} \rightarrow c_{\lam}~~\mathrm{a.s.} & \mathrm{if} & \frac{d}{n} \rightarrow \infty.
\end{eqnarray}
\end{lemma}
\begin{proof}
Let $\tilYsub{i}{1}=\big( y_{ij}I_{\{j \in s_i^c \}}\big)_{1 \leq j \leq n}$ be an $n$-dimensional vector whose elements are $y_{ij}$ if $j \in s_i^c$ and $0$ if $j \in s_i$ and $\tilYsub{i}{2}=\tilY_i - \tilYsub{i}{1}$. Similarly, define $\tilZsub{i}{1}$ and $\tilZsub{i}{2}$ with $\tilZ_i = (z_{i1}, \cdots, z_{in})^{T}$. Then,
\begin{eqnarray}\label{eq:B}
	\B & = & \frac{1}{n}\sum_{i=K+1}^{d}\tilY_i \tilY_i^{T} \nonumber \\
	& = & \frac{1}{n}\sum_{i=K+1}^{d}\big( \tilYsub{i}{1}\tilYsub{i}{1}^{T} +  \tilYsub{i}{2}\tilYsub{i}{2}^{T} \big) \nonumber \\
	& = & \frac{1}{n}\sum_{i=K+1}^{d}\big( \taun{i,1} \tilZsub{i}{1}\tilZsub{i}{1}^{T} +  \taun{i,2} \tilZsub{i}{2}\tilZsub{i}{2}^{T} \big)	
\end{eqnarray}
Without loss of generality, assume that $\taun{i,1} \leq \taun{i,2}$ for all $n$ and all $i=K+1, \cdots, d$. Then,
\begin{eqnarray*}
	\B + \frac{1}{n} \sum_{i=K+1}^{d} (\taun{i,2}-\taun{i,1})\tilZsub{i}{1}\tilZsub{i}{1}^{T} = \frac{1}{n}\sum_{i=K+1}^{d} \taun{i,2}\tilZ_i \tilZ_i^{T}.
\end{eqnarray*}  
It follows from the Weyl inequality that 
\begin{eqnarray*}
	\eigval{k}{\B} + \eigval{n}{\frac{1}{n} \sum_{i=K+1}^{d} (\taun{i,2}-\taun{i,1})\tilZsub{i}{1}\tilZsub{i}{1}^{T}} \leq \eigval{k}{\frac{1}{n}\sum_{i=K+1}^{d} \taun{i,2}\tilZ_i \tilZ_i^{T}}. 
\end{eqnarray*}
Note that $\frac{1}{n} \sum_{i=K+1}^{d} (\taun{i,2}-\taun{i,1})\tilZsub{i}{1}\tilZsub{i}{1}^{T}$ is nonnegative definite and thus we have
\begin{eqnarray}\label{eq:B_rightineq_1}
	\eigval{k}{\B} \leq \eigval{k}{\frac{1}{n}\sum_{i=K+1}^{d} \taun{i,2}\tilZ_i \tilZ_i^{T}}. 	
\end{eqnarray}
Also, since 
\begin{eqnarray*}
	\frac{1}{n}\sum_{i=K+1}^{d} \taun{i,2}\tilZ_i Z_i^{T} + \frac{1}{n}\sum_{i=K+1}^{d} (\max_{K+1 \leq i \leq d}\taun{i,2} - \taun{i,2})\tilZ_i \tilZ_i^{T} = \frac{1}{n}\sum_{i=K+1}^{d} \max_{K+1 \leq i \leq d}\taun{i,2} \tilZ_i \tilZ_i^{T}. 
\end{eqnarray*}
Again, the Weyl inequality and the nonnegativity of $\frac{1}{n}\sum_{i=K+1}^{d} (\max_{K+1 \leq i \leq d}\taun{i,2} - \taun{i,2})\tilZ_i \tilZ_i^{T}$ yield the following inequality
\begin{eqnarray*}
	\eigval{k}{\frac{1}{n}\sum_{i=K+1}^{d} \taun{i,2}\tilZ_i \tilZ_i^{T}} \leq \eigval{k}{\frac{1}{n}\sum_{i=K+1}^{d} \max_{K+1 \leq i \leq d}\taun{i,2} \tilZ_i \tilZ_i^{T}}.
\end{eqnarray*} 
Hence, it follows from (\ref{eq:B_rightineq_1}) that 
\begin{eqnarray}\label{eq:B_rightineq}
	\eigval{k}{\B} \leq \eigval{k}{\frac{1}{n}\sum_{i=K+1}^{d} \max_{K+1 \leq i \leq d}\taun{i,2} \tilZ_i \tilZ_i^{T}}.
\end{eqnarray}
To get a lower bound of $\eigval{k}{B}$, we start with the following equality
\begin{eqnarray*}
		\B = \frac{1}{n}\sum_{i=K+1}^{d} \taun{i,1}\tilZ_i \tilZ_i^{T} + \frac{1}{n} \sum_{i=K+1}^{d} (\taun{i,2}-\taun{i,1})\tilZsub{i}{2}\tilZsub{i}{2}^{T} .
\end{eqnarray*}
from (\ref{eq:B}). Then, we obtain 
\begin{eqnarray}\label{eq:B_leftineq}
	\eigval{k}{\frac{1}{n}\sum_{i=K+1}^{d} \min_{K+1 \leq i \leq d}\taun{i,1} \tilZ_i \tilZ_i^{T}} \leq \eigval{k}{\B}
\end{eqnarray}
in a similar way to get (\ref{eq:B_rightineq}). 

By (\ref{eq:B_rightineq}) and (\ref{eq:B_leftineq}) and letting $\B^{*}=\frac{1}{n}\sum_{i=K+1}^{d}  \tilZ_i \tilZ_i^{T}$, it follows that
\begin{eqnarray}\label{eq:B_ineq}
	\min_{K+1 \leq i \leq d}\taun{i,1}\times \eigval{k}{\B^*} \leq \eigval{k}{\B} \leq \max_{K+1 \leq i \leq d}\taun{i,2} \times \eigval{k}{\B^*}.
\end{eqnarray}
By Lemma \ref{lem:Bai-Yin}, the convergence of the extreme eigenvalues of $\B^*$ can be obtained as follows:
\begin{eqnarray}\label{eq:B*conv}
	\eigval{\max}{\B^*} \rightarrow (1+\sqrt{c})^2 ~\mathrm{and}~\eigval{\min}{\B^*} \rightarrow (1-\sqrt{c})^2 & \mathrm{if} &\frac{d}{n} \rightarrow c < \infty \nonumber \\
	\frac{n}{d}\eigval{\max}{\B^*}~\mathrm{and}~\frac{n}{d}\eigval{\min}{\B^*} \rightarrow 1~~\mathrm{a.s.} & \mathrm{if} & \frac{d}{n} \rightarrow \infty. 
\end{eqnarray}
For details, see the proof of Lemma 6.4 in \cite{shen2016general}.

Since $\min_{K+1 \leq i \leq d}\taun{i,1}$ and $\max_{K+1 \leq i \leq d} \taun{i,2} \rightarrow c_{\lam}$ by Assumption \ref{assume:nonspike}, it follows from (\ref{eq:B_ineq}) and (\ref{eq:B*conv}) that 
\begin{eqnarray*}
	\eigval{\max}{\B} \rightarrow c_{\lam}(1+\sqrt{c})^2 ~ \mathrm{and} ~ \eigval{\min}{\B} \rightarrow c_{\lam}(1-\sqrt{c})^2~~\mathrm{a.s.} & \mathrm{if} & \frac{d}{n} \rightarrow c < \infty \nonumber \\
	\frac{n}{d}\eigval{\max}{\B} ~ \mathrm{and}~ \frac{n}{d}\eigval{\min}{\B} \rightarrow c_{\lam}~~\mathrm{a.s.} & \mathrm{if} & \frac{d}{n} \rightarrow \infty.
\end{eqnarray*}	
\end{proof}
\textit{Theorem \ref{thm1:eigenvalue} (a).} So far, we proved the first and second steps and now we prove the last step that completes the proof. It follows from the Weyl inequality and (\ref{eq:dual_AB}) that 
\begin{eqnarray}\label{eq:sandwich}
	\eigval{i}{\A} + \eigval{n}{\B} \leq \eigval{i}{\dual{\samcov}} \leq \eigval{i}{\A} + \eigval{1}{\B}~~\mathrm{for}~i=1, \cdots, \min(d,n).
\end{eqnarray}
Let us first consider the case when $0<c<\infty$. Then, the condition $\frac{d}{n\deln_M} \rightarrow 0$ yields $\deln_M \rightarrow \infty $, and thus $\frac{1}{\lamn_i}\eigval{n}{\B}$ and $\frac{1}{\lamn_i}\eigval{1}{\B} \rightarrow 0$ for $i \leq K$ by Lemma \ref{lem:conv_B}. According to Lemma \ref{lem:conv_A}, we conclude that for $i \leq K$ 
\begin{eqnarray}\label{eq:thm_main_a1}
	\frac{\hatlam_i}{\lamn_i} \rightarrow~ 1~~\mathrm{a.s.}
\end{eqnarray}
because $\hatlam_i = \eigval{i}{\samcov}=\eigval{i}{\dual{\samcov}}$.
For $i > K$, $\eigval{i}{\A}=0$ since the rank of $\A$ is less than or equal to $K$. Then, by Lemma \ref{lem:conv_B}, 
\begin{eqnarray}\label{eq:thm_main_a1_nonspike}
	c_{\lam}(1-\sqrt{c})^2 \leq \liminf \hatlam_i \leq \limsup \hatlam_i \leq c_{\lam}(1+\sqrt{c})^2 ~~\mathrm{for}~i > K. 
\end{eqnarray}

Next, consider the case when $c=\infty$. By the condition $\frac{d}{n\deln_M} \rightarrow 0$, $\frac{1}{\lamn_i}\eigval{n}{\B}=\frac{n}{d}\eigval{n}{\B}\times \frac{d}{n\deln_M} \times \frac{\deln_M}{\lamn_i} \rightarrow 0$ for $i \leq K$. Similarly, $\frac{1}{\lamn_i}\eigval{1}{\B} \rightarrow 0$ for $i \leq K$. Thus, we can conclude (\ref{eq:thm_main_a1}) for $i \leq K$. For $i > K$, it follows from $\eigval{i}{\A}=0$ and Lemma \ref{lem:conv_B} that 
\begin{eqnarray*}
\frac{n}{d}\hatlam_{K+1} \rightarrow c_{\lam} ~~~\mathrm{and}~~~\frac{n}{d}\hatlam_n \rightarrow c_{\lam}~~\mathrm{a.s.},
\end{eqnarray*}
which gives
\begin{eqnarray*}
\frac{n}{d}\hatlam_i \rightarrow c_{\lam}~~\mathrm{a.s.} ~~~\mathrm{for}~~i > K.
\end{eqnarray*}

Lastly, consider the case when $c=0$. In this case, the condition $\frac{d}{n\deln_M} \rightarrow 0$ does not guarantee $\deln_M \rightarrow \infty $ so that we divide the case into two sub-cases: $\deln_M \rightarrow \infty $ and $\deln_M < \infty $. When $\deln_M \rightarrow \infty $, (\ref{eq:thm_main_a1}) follows similarly to the case when $0<c<\infty$. When $\deln_M < \infty $, according to Theorem 1 ($c=0$) of \cite{baik2006eigenvalues} as mentioned in \cite{shen2016general}, (\ref{eq:thm_main_a1}) still follows for $i \leq K$. For $i > K$, it is easy to see (\ref{eq:thm_main_a1_nonspike}) with $c=0$, that is, 
\begin{eqnarray*}
	c_{\lam} \leq \liminf \hatlam_i \leq \limsup \hatlam_i \leq c_{\lam} ~~\mathrm{for}~i > K.	
\end{eqnarray*}
Hence, we have $\hatlam_i \rightarrow c_{\lam}$ for $i>K$. This completes the proof. 

\textit{Theorem \ref{thm1:eigenvalue} (b).} As mentioned earlier, the condition $\frac{d}{n\deln_{h+1}} \rightarrow \infty$ implies $\frac{d}{n} \rightarrow \infty$. So, we only consider the case $\frac{d}{n} \rightarrow \infty$. 

\begin{itemize}
    \item[(i)] $i \leq p_h$: since $\frac{d}{n\deln_{h}} \rightarrow 0$, $\deln_{h} \rightarrow \infty$ and thus $\frac{\eigval{1}{\B}}{\lamn_i} \leq \frac{\eigval{1}{\B}}{\deln_h} \times \frac{\deln_h}{\lamn_{p_h}} \rightarrow_{a.s.} 0$.  Similarly, $\frac{\eigval{n}{\B}}{\lamn_i}  \rightarrow_{a.s.} 0$.  Then, (\ref{eq:thm_main_a1}) follows from (\ref{eq:sandwich}) for $i \leq p_h$. 
    
    \item[(ii)] $i > p_h$: (\ref{eq:sandwich}) can be re-expressed as 
    \begin{eqnarray*}
        \frac{n}{d}\eigval{i}{\A} + \frac{n}{d}\eigval{n}{\B} \leq \frac{n}{d}\hatlam_i \leq \frac{n}{d}\eigval{i}{\A} + \frac{n}{d}\eigval{1}{\B}.  
    \end{eqnarray*}
    We can easily see that $\frac{n}{d} \eigval{i}{\A} \leq \frac{n}{d} \eigval{p_h+1}{\A} = \frac{n\deln_{h+1}}{d} \times \frac{\eigval{p_h+1}{\A}}{\deln_{h+1}} \rightarrow_{a.s.} 0$ by the condition $\frac{d}{n\deln_{h+1}} \rightarrow \infty$. Therefore, it follows from Lemma \ref{lem:conv_B} that $\frac{n}{d}\hatlam_i \rightarrow_{a.s.} c_{\lam}$ for $i>p_h$.

\end{itemize}
\end{proof}

\subsection{Proof of Theorem \ref{thm2:eigenvector}}
\begin{proof}
\indent For the subspace consistency of the sample eigenvectors, $\hatU_i$ to the $S_m$, we want to show 
\begin{eqnarray}\label{eq:angle}
    \mathrm{angle}(\hatU_i, S_m) \rightarrow_{a.s.} 0. 
\end{eqnarray}
This is equivalent to showing that $\cos(\mathrm{angle}(\hatU_i, S_m))=\big( \sum_{k \in H_m} \hatU_i^{T} U_k U_k^{T}\hatU_i\big)^{1/2} \rightarrow_{a.s.} 1$ for $i \in H_m$ since $S_m = \mathrm{span}\{U_i\}_{i \in H_m}$. Without loss of generality, we can assume $U_k=e_k$ where the $k$th entry is 1 and the rest of entries are all zero. Then, $\hatU_i^{T}U_k$ is simply $\hatuki$ and therefore (\ref{eq:angle}) is equivalent to
\begin{eqnarray}\label{eq:subspace_sumsq}
    \sum_{k \in H_m}\hatuki^2 \rightarrow_{a.s.} 1~~~~~~\mathrm{for}~i\in H_m. 
\end{eqnarray}
In general, we will show (\ref{eq:subspace_sumsq}) for the subspace consistency, but different convergence rates will be achieved under each scenario. For the strong inconsistency of $\hatU_i$ to the true eigenvector $U_i$, we will show that the $\mathrm{angle}(\hatU_i, U_i) \rightarrow_{a.s.} \frac{\pi}{2}$, which is equivalent to showing that $\cos(\mathrm{angle}(\hatU_i, U_i))=\hatu_{ii} \rightarrow_{a.s.} 0$.

Before we prove the main parts, we first provide some important results that will be used in the proof of the main parts. Define $\Sbf=\matrootinv{\Lam} \hatUbf \matroot{\hatLambf}$ where $\Lam=\mathrm{diag}(\lamn_1, \cdots, \lamn_d) $, and then its element for is the $k$th row and $i$th column is $s_{ki} = \frac{\sqrt{\hatlam_i}}{\sqrt{\lamn_k}}\hatuki$. From (\ref{eq:mixture}), we obtain $\Sbf \Sbf^{T} = \frac{1}{n}\matrootinv{\Lam} \Y \Y^{T}\matrootinv{\Lam}$, and thus the $k$th diagonal element of $\Sbf \Sbf^{T}$, i.e.
\begin{eqnarray*}
(\Sbf \Sbf^{T})_{(k,k)} = \sumi s_{ki}^2 = \sumi \frac{\hatlam_i}{\lamn_k} \hatuki^2 = \frac{1}{\lamn_k}\sumi \hatlam_i \hatuki^2,
\end{eqnarray*}
 is equal to the $k$th diagonal element of $\frac{1}{n}\matrootinv{\Lam} \Y \Y^{T}\matrootinv{\Lam}=\meanj \frac{1}{\lamn_k}y_{kj}^2$. From (\ref{eq:diagonal}) and (\ref{eq:leftconv}), we get
\begin{eqnarray}\label{eq:S1}
    \frac{1}{\lamn_k}\sumi \hatlam_i \hatuki^2 = \meanj \frac{1}{\lamn_k}y_{kj}^2 \rightarrowas 1.
\end{eqnarray}

Since all diagonal values of a matrix should be less than its largest eigenvalue, it follows from $\lam_i(\Sbf\Sbf^{T}) = \lam_i(\Sbf^{T}\Sbf)$ that
\begin{eqnarray}\label{eq:S2}
    (\Sbf^{T}\Sbf)_{(i,i)} = \sum_{k=1}^{d} s_{ki}^2 = \sum_{k=1}^{d} \frac{\hatlam_i}{\lamn_k}\hatuki^2 = \hatlam_i \sum_{k=1}^{d} \frac{1}{\lamn_k} \hatuki^2 \leq \lam_1(\frac{1}{n}\matrootinv{\Lam}\Y \Y^{T}\matrootinv{\Lam}).
\end{eqnarray}
To use the above inequality in the proof of the main parts, the largest eigenvalue of $\frac{1}{n}\matrootinv{\Lam}\Y \Y^{T}\matrootinv{\Lam}$ is of interest. If there are no outlier components, $\matrootinv{\Lam}\Y$ will consist of i.i.d. random variables with zero mean, unit variance, and finite fourth moment. Then, Lemma \ref{lem:Bai-Yin} implies that 
\begin{eqnarray*}
\lam_1(\frac{1}{n}\matrootinv{\Lam}\Y \Y^{T}\matrootinv{\Lam}) & \rightarrowas & (1+\sqrt{c})^2 ~~~\mathrm{for} ~0 \leq c < \infty~~~\mathrm{and} \\
\lam_1(\frac{1}{d}\matrootinv{\Lam}\Y \Y^{T}\matrootinv{\Lam}) & \rightarrowas & (1+\frac{1}{\sqrt{c}})^2~~~\mathrm{for}~0 < c \leq \infty. 	
\end{eqnarray*}
In the case with outliers from the model (\ref{eq:mixture}), however, the entries of the $\matrootinv{\Lam}\Y$ are not identically distributed any longer, which brings same challenges. Here, we prove that the maximum eigenvalue of $\frac{1}{n}\matrootinv{\Lam}\Y \Y^{T}\matrootinv{\Lam}$ still has the same limit even though there are a few outliers from different distributions. 

\begin{lemma}\label{lem:max_eigenvalue} As $n,~d \rightarrow \infty$ such that $\frac{d}{n} \rightarrow c $, we have
\begin{eqnarray*}
 \lam_1 (\frac{1}{n}\matrootinv{\Lam}\Y \Y^{T}\matrootinv{\Lam}) \rightarrowas & (1+\sqrt{c})^2 ~~~~& 0 \leq \mathrm{c} < \infty \\
 \lam_1 (\frac{1}{d} \matrootinv{\Lam}\Y \Y^{T}\matrootinv{\Lam}) \rightarrowas & (1+\frac{1}{\sqrt{c}})^2 ~~~~& 0 < \mathrm{c} \leq \infty. 
\end{eqnarray*}
\end{lemma}
\begin{proof}
Denote the $i$th row vector of $\matrootinv{\Lam}\Y$ by $\checkYi$. Then, $\frac{1}{d}\Y^T \Lam^{-1}\Y = \meani \checkYi \checkYi^{T} = \frac{1}{d}\sum_{i=1}^{K} \checkYi \checkYi^{T} + \frac{1}{d}\sum_{i>K}\checkYi \checkYi^{T}$. Recall that each element of $\checkYi$ is $\breve{y}_{ij} = \tauifn z_{ij}$ for $j \in s_{i}^c$ and $\breve{y}_{ij} = \tauisn z_{ij}$ for $j \in s_{i}$. Let $Z_i = (z_{i1}, \cdots, z_{in})^{T}$ for $i = 1, \cdots, d$, and then we have 
\begin{eqnarray}
    \frac{1}{d}\Y^T \Lam^{-1}\Y = \frac{1}{d}\Big( \sum_{i=1}^{K} \checkYi \checkYi^{T} - \sum_{i=1}^{K}Z_i Z_i^{T} \Big) + \frac{1}{d}\Big( \sum_{i=1}^{K}Z_i Z_i^{T} + \sum_{i=K+1}^{d} \checkYi \checkYi^{T}  \Big).
\end{eqnarray}
By the Weyl inequality, we have 
\begin{eqnarray}\label{eq:lem5_upper_main}
    \lam_1 (\meani \checkYi \checkYi^{T}) \leq \lam_1 \Big(\frac{1}{d}\Big( \sum_{i=1}^{K} \checkYi \checkYi^{T} - \sum_{i=1}^{K}Z_i Z_i^{T} \Big) \Big) + \lam_1 \Big(\frac{1}{d}\Big( \sum_{i=1}^{K}Z_i Z_i^{T} + \sum_{i=K+1}^{d} \checkYi \checkYi^{T}  \Big)\Big).
\end{eqnarray}
Letting $\checkYbf_K = [\checkY_1, \cdots, \checkY_K]$ and $\Z_K = [Z_1, \cdots, Z_K]$, we get $\frac{1}{d}\Big( \sum_{i=1}^{K} \checkYi \checkYi^{T} - \sum_{i=1}^{K}Z_i Z_i^{T} \Big) = \frac{1}{d}\Big( \checkYbf_K \checkYbf_K^{T} - \Z_K \Z_K^{T} \Big) $. Then, we have 
\begin{eqnarray*}
\lam_1 \Big(\frac{1}{d}\Big( \checkYbf_K \checkYbf_K^{T} - \Z_K \Z_K^{T} \Big) \Big) & \leq & \tr \Big( \frac{1}{d}\Big( \checkYbf_K \checkYbf_K^{T} - \Z_K \Z_K^{T} \Big) \Big) \\
& = &\frac{1}{d} \Big[ \tr \Big( \checkYbf_K \checkYbf_K^{T}\Big) - \tr \Big( \Z_K \Z_K^{T} \Big) \Big] \\
& = & \frac{1}{d} \Big[ \tr \Big( \checkYbf_K^{T} \checkYbf_K\Big) - \tr \Big( \Z_K^{T} \Z_K \Big) \Big] \\
& = & \frac{1}{d} \Big[ \sum_{i=1}^{K}\sumj\frac{1}{\lamn_i}y_{ij}^2 - \sum_{i=1}^{K}\sumj z_{ij}^2\Big] \\
& = & \frac{n}{d} \Big[ \sum_{i=1}^{K}\meanj\frac{1}{\lamn_i}y_{ij}^2 - \sum_{i=1}^{K}\meanj z_{ij}^2 \Big].    
\end{eqnarray*}
Because $\sum_{i=1}^{K}\meanj\frac{1}{\lamn_i}y_{ij}^2 - \sum_{i=1}^{K}\meanj z_{ij}^2 \rightarrowas 0$ and $\frac{n}{d} \rightarrow \frac{1}{c}$, we have 
\begin{eqnarray}\label{eq:lem5_upper_eq1}
    \lam_1 \Big(\frac{1}{d}\Big( \sum_{i=1}^{K} \checkYi \checkYi^{T} - \sum_{i=1}^{K}Z_i Z_i^{T} \Big) \Big) \rightarrowas 0~~~~\mathrm{for}~0 \leq \frac{1}{c} < \infty.
\end{eqnarray}
Based on techniques similar to the proof of Lemma \ref{lem:conv_B}, one can show that 
\begin{eqnarray}\label{eq:lem5_upper_eq2}
    \lam_1 \Big(\frac{1}{d}\Big( \sum_{i=1}^{K}Z_i Z_i^{T} + \sum_{i=K+1}^{d} \checkYi \checkYi^{T}  \Big)\Big) \rightarrowas (1+\frac{1}{\sqrt{c}})^2.
\end{eqnarray}
By (\ref{eq:lem5_upper_main}), (\ref{eq:lem5_upper_eq1}), and (\ref{eq:lem5_upper_eq2}), we have
\begin{eqnarray}\label{eq:lem5_upper_result}
    \lam_1 (\meani \checkYi \checkYi^{T}) \leq (1+\frac{1}{\sqrt{c}})^2~~~~\mathrm{almost~surely}.
\end{eqnarray}
For the lower bound, application of Weyl inequality on the other way gives
\begin{eqnarray*}
  \lam_n \Big(\frac{1}{d}\Big( \sum_{i=1}^{K} \checkYi \checkYi^{T} - \sum_{i=1}^{K}Z_i Z_i^{T} \Big) \Big) + \lam_1 \Big(\frac{1}{d}\Big( \sum_{i=1}^{K}Z_i Z_i^{T} + \sum_{i=K+1}^{d} \checkYi \checkYi^{T}  \Big)\Big)  \leq \lam_1 (\meani \checkYi \checkYi^{T}).
\end{eqnarray*}
Since the rank of $\frac{1}{d}\Big( \sum_{i=1}^{K} \checkYi \checkYi^{T} - \sum_{i=1}^{K}Z_i Z_i^{T} \Big)$ is less than or equal to $K$, we have $\lam_n\Big(\frac{1}{d}\Big( \sum_{i=1}^{K} \checkYi \checkYi^{T} - \sum_{i=1}^{K}Z_i Z_i^{T} \Big)\Big)=0$, which with (\ref{eq:lem5_upper_eq2}) gives 
\begin{eqnarray}\label{eq:lem5_lower_result}
    (1+\frac{1}{\sqrt{c}})^2 \leq \lam_1 (\meani \checkYi \checkYi^{T})~~~~\mathrm{almost~surely}.
\end{eqnarray}
A combination of (\ref{eq:lem5_upper_result}) and (\ref{eq:lem5_lower_result}) completes the proof. 
\end{proof}

Now we start to prove Theorem \ref{thm2:eigenvector}. We will first prove (b) and move on to (a). 

\textbf{proof of (b).} Assumption \ref{assume:tier_rates} and $\frac{d}{n\deln_{h+1}} \rightarrow \infty$ together imply $\dn \rightarrow \infty$. So here we only consider the case of $\dn \rightarrow \infty$. The proof consists of the following three steps:

\begin{enumerate}
    \item Establish the convergence for the $h$-th tier.
    \item Establish the convergence for the $m$-th tier sequentially from $m=h-1$ to $1$.
    \item Establish the strong inconsistency of the remaining sample eigenvectors.
\end{enumerate}
We start with the first step.

1. Establish the convergence for the $h$-th tier, i.e. $\mathrm{angle}(\hat{U}_i, S_h) = o(\big\{ \delfrac{h}{h-1} \big\}^{1/2}) \vee O(\big\{ \frac{d}{n\deln_{h}} \big\}^{1/2})$ for $i \in H_h$. As discussed earlier, we need to show the following: 
\begin{eqnarray}\label{eq:thm2_b_step1_main}
    \sum_{k \in H_h} \hatuki^2 = 1+o\Big( \delfrac{h}{h-1} \Big) \vee O\Big( \frac{d}{n\deln_{h}} \Big) ~~~~\mathrm{for}~i \in H_h. 
\end{eqnarray}
This can be proved by showing the following two equations:
\begin{eqnarray}\label{eq:thm2_bi_bullet2_eq1}
    \sum_{i=p_h+1}^{d}\hatuki^2 = \bigo{\frac{d}{n\deln_h}}~~~~\mathrm{for}~k \in H_h
\end{eqnarray}
and
\begin{eqnarray}\label{eq:thm2_bi_bullet2_lasteq3}
    \sum_{m=1}^{h-1}\sum_{i \in H_m}\hatuki^2 = O\Big( \frac{\deln_{h}}{\deln_{h-1}} \Big)~~~~\mathrm{for}~k \in H_h.
\end{eqnarray}
Since $\sum_{i=p_h+1}^{d}\hatuki^2  \leq \sum_{k=1}^{p_h}\sum_{i=p_h+1}^{d}\hatuki^2 = \sum_{k=p_h+1}^{d}\sum_{i=1}^{p_h}\hatuki^2 $ and $p_h$ is finite, the first equation (\ref{eq:thm2_bi_bullet2_eq1}) is equivalent to
\begin{eqnarray}\label{eq:thm2_bi_bullet2_eq2}
    \sum_{k=p_{h}+1}^{d}\sum_{i=1}^{p_h}\hatuki^2 = O\Big( \frac{d}{n\deln_h} \Big)~~~~\mathrm{for}~m=1, \cdots, h    
\end{eqnarray}
for $m=h$. The proof of (\ref{eq:thm2_bi_bullet2_eq2}) is equivalent to showing the following two equations:
\begin{eqnarray}\label{eq:thm2_bi_bullet2_lasteq1}
    \sum_{k=K+1}^{d}\sum_{i=1}^{p_m}\hatuki^2 = O\Big( \frac{d}{n\deln_m} \Big)~~~~\mathrm{for}~m=1, \cdots, h    
\end{eqnarray}
and
\begin{eqnarray}\label{eq:thm2_bi_bullet2_lasteq2}
    \sum_{k=p_h+1}^{K}\sum_{i=1}^{p_m}\hatuki^2 = O\Big( \frac{\deln_{h+1}}{\deln_{m}} \Big)~~~~\mathrm{for}~m=1, \cdots, h.
\end{eqnarray}
For details, see \cite{shen2016general}. Thus, in order to show (\ref{eq:thm2_b_step1_main}), it is enough to show (\ref{eq:thm2_bi_bullet2_lasteq3}), (\ref{eq:thm2_bi_bullet2_lasteq1}), (\ref{eq:thm2_bi_bullet2_lasteq2}). We start with (\ref{eq:thm2_bi_bullet2_lasteq1}).

\emph{proof of (\ref{eq:thm2_bi_bullet2_lasteq1}).} For $m=1, \cdots, h$, from $\frac{\hatlam_{p_m}}{\lamn_{K+1}} \sum_{k=K+1}^{d}\sum_{i=1}^{p_m}\hatuki^2 \leq \sum_{k=K+1}^{d}\frac{1}{\lamn_k}\sum_{i=1}^{p_m}\hatlam_{i}\hatuki^2$,
\begin{eqnarray}\label{eq:thm2_bi_bullet2_lasteq1_bound}
     \sum_{k=K+1}^{d}\sum_{i=1}^{p_m}\hatuki^2 & \leq & \frac{\lamn_{K+1}}{\hatlam_{p_m}}\sum_{i=1}^{p_m}\hatlam_{i}\sum_{k=K+1}^{d}\frac{1}{\lamn_k}\hatuki^2  \nonumber \\
     & \leq & \frac{d}{n\deln_m} \frac{\deln_m}{\hatlam_{p_m}}\frac{n\lamn_{K+1}}{d}\sum_{i=1}^{p_m}\hatlam_i \sum_{k=1}^{d}\frac{1}{\lamn_k} \hatuki^2 \nonumber \\
     & \leq & \frac{d}{n\deln_m} \frac{\deln_m}{\hatlam_{p_m}}\frac{n\lamn_{K+1}}{d}p_m \lam_1(\frac{1}{n}\matrootinv{\Lam}\Y \Y^{T}\matrootinv{\Lam})~~~~\mathrm{by}~(\ref{eq:S2}) \nonumber \\
     & = & \frac{d}{n\deln_m} \frac{\deln_m}{\hatlam_{p_m}}\lamn_{K+1}p_m \lam_1(\frac{1}{d}\matrootinv{\Lam}\Y \Y^{T}\matrootinv{\Lam})
\end{eqnarray}
By Theorem \ref{thm1:eigenvalue}(b), Assumption \ref{assume:nonspike}, and Lemma \ref{lem:max_eigenvalue}, for large $n$ and $d$, the last expression of (\ref{eq:thm2_bi_bullet2_lasteq1_bound}) becomes a constant multiplied by $\frac{d}{n\deln_m}$ that converges to 0. Therefore, the proof of (\ref{eq:thm2_bi_bullet2_lasteq1}) is complete.

\emph{proof of (\ref{eq:thm2_bi_bullet2_lasteq2}).} For $m=1, \cdots, h$, from $\frac{\hatlam_{p_m}}{\lamn_{p_h+1}}\sum_{k=p_h+1}^{K}\sum_{i=1}^{p_m}\hatuki^2 \leq \sum_{k=p_h+1}^{K} \frac{1}{\lamn_k} \sum_{i=1}^{p_m}\hatlam_{i}\hatuki^2$, 
\begin{eqnarray}\label{eq:thm2_bi_bullet2_lasteq2_bound}
    \sum_{k=p_h+1}^{K}\sum_{i=1}^{p_m}\hatuki^2 & \leq &  \frac{\lamn_{p_h+1}}{\hatlam_{p_m}} \sum_{k=p_h+1}^{K} \frac{1}{\lamn_k} \sum_{i=1}^{p_m}\hatlam_{i}\hatuki^2 \nonumber \\
    & \leq & \frac{\lamn_{p_h+1}}{\hatlam_{p_m}}\sum_{k=p_h+1}^{K} \meanj \frac{1}{\lamn_k}y_{kj}^2 ~~~~\mathrm{by}~(\ref{eq:S1}) \nonumber \\
    & = & \frac{\lamn_{p_h+1}}{\deln_{h+1}}\delfrac{h+1}{m}\frac{\deln_m}{\lamn_{p_m}}\frac{\lamn_{p_m}}{\hatlam_{p_m}}\sum_{k=p_h+1}^{K} \meanj \frac{1}{\lamn_k}y_{kj}^2
\end{eqnarray}
For $m=h$, it follows from $\delfrac{h+1}{h} \rightarrowas 0$, Theorem \ref{thm1:eigenvalue}(b), and (\ref{eq:S1}), the proof of (\ref{eq:thm2_bi_bullet2_lasteq2}) is complete. 

\emph{proof of (\ref{eq:thm2_bi_bullet2_lasteq3}).} From (\ref{eq:S1}), we obtain 
\begin{eqnarray}\label{eq:thm2_bi_bullet2_lasteq3_eq1}
    \frac{1}{\lamn_k}\sumi \hatlam_i \hatuki^2 & = & \frac{1}{\lamn_k}\sum_{i=1}^{p_{h-1}} \hatlam_i \hatuki^2 + \frac{1}{\lamn_k}\sum_{i \in H_h} \hatlam_i \hatuki^2 + \frac{1}{\lamn_k}\sum_{i=p_h+1}^{d} \hatlam_i \hatuki^2 \nonumber \\
    & \rightarrowas & 1~~~~\mathrm{for}~k \in H_h.
\end{eqnarray}
The third part of (\ref{eq:thm2_bi_bullet2_lasteq3_eq1}) 
\begin{eqnarray}
    \frac{1}{\lamn_k}\sum_{i=p_h+1}^{d} \hatlam_i \hatuki^2 & \leq & \frac{\hatlam_{p_h+1}}{\lamn_{p_h}} \sum_{i=p_h+1}^{d} \hatuki^2 \nonumber \\
    & = & \frac{\deln_{p_h}}{\lamn_{p_h}}\frac{n\hatlam_{p_h+1}}{d}\frac{d}{n\deln_{p_h}}\sum_{i=p_h+1}^{d} \hatuki^2 \nonumber \\
    & \rightarrowas & 0 ~~~~\mathrm{for}~k \in H_h.
\end{eqnarray}
The convergence follows from Theorem \ref{thm1:eigenvalue} (b) and (\ref{eq:thm2_bi_bullet2_eq1}). Thus, (\ref{eq:thm2_bi_bullet2_lasteq3_eq1}) becomes 
\begin{eqnarray}\label{eq:thm2_bi_bullet2_lasteq3_eq2}
    \frac{1}{\lamn_k}\sum_{i=1}^{p_{h-1}} \hatlam_i \hatuki^2 + \frac{1}{\lamn_k}\sum_{i \in H_h} \hatlam_i \hatuki^2  \rightarrowas 1~~~~\mathrm{for}~k \in H_h.
\end{eqnarray}
Because $\frac{\hatlam_i}{\lamn_k} \rightarrowas \delfrac{m}{h}$ for $i \in H_m$ and $k \in H_h$, (\ref{eq:thm2_bi_bullet2_lasteq3_eq2}) can be rewritten as
\begin{eqnarray}\label{eq:thm2_bi_bullet2_lasteq3_eq3}
    \sum_{m=1}^{h-1}\sum_{i \in H_m} \delfrac{m}{h} \hatuki^2 + \frac{1}{\lamn_k}\sum_{i \in H_h} \hatlam_i \hatuki^2  \rightarrowas 1~~~~\mathrm{for}~k \in H_h.
\end{eqnarray}
Also, since $\sumi \hatuki^2 = 1$ and $\sum_{i=p_h+1}^{d} \hatuki^2 \rightarrowas 0$ for $k \in H_h$ by (\ref{eq:thm2_bi_bullet2_eq1}), we have
\begin{eqnarray}\label{eq:thm2_bi_bullet2_lasteq3_eq4}
    \sum_{m=1}^{h-1}\sum_{i \in H_m} \hatuki^2 + \frac{1}{\lamn_k}\sum_{i \in H_h} \hatlam_i \hatuki^2  \rightarrowas 1~~~~\mathrm{for}~k \in H_h.
\end{eqnarray}
From (\ref{eq:thm2_bi_bullet2_lasteq3_eq3}), (\ref{eq:thm2_bi_bullet2_lasteq3_eq4}), and $\delfrac{m}{h} \geq 1$, we get $\sum_{m=1}^{h-1}\sum_{i \in H_m} \hatuki^2 \rightarrowas 0$ and $\sum_{i \in H_h} \hatuki^2 \rightarrowas 1$. This already shows the subspace consistency, and for convergence rate of (\ref{eq:thm2_bi_bullet2_lasteq3}), $\sum_{m=1}^{h-1}\sum_{i \in H_m} \hatuki^2 \leq \sum_{m=1}^{h-1}\sum_{i \in H_m}\delfrac{m}{h} \hatuki^2 \leq (h-1) \delfrac{h-1}{h} \sum_{i \in H_m} \hatuki^2 \rightarrowas 0$ for $m=1, \cdots, h-1$. Hence, we complete the proof of the step 1. We move on to Step 2.

2. Establish the convergence for the $m$-th tier sequentially from $m=h-1$ to $1$, i.e. $\mathrm{angle}(\hat{U}_i, S_m) = o(\big\{ \frac{\deln_{m}}{\deln_{m-1}} \vee \frac{\deln_{m+1}}{\deln_{m}} \big\}^{1/2})$ for $i \in H_m$ for each $m$.

We want to show 
\begin{eqnarray*}
    \sum_{k \in H_m} \hatuki^2 = 1 + \littleo{\delfrac{m}{m-1} \vee \delfrac{m+1}{m}}~~~~\mathrm{for}~i \in H_m,~m=1, \cdots, h-1,
\end{eqnarray*}
which is equivalent to showing
\begin{eqnarray}\label{eq:thm2_b_step2_main}
    \sum_{i \in H_m} \hatuki^2 = 1 + \littleo{\delfrac{m}{m-1} \vee \delfrac{m+1}{m}}~~~~\mathrm{for}~k \in H_m,~m=1, \cdots, h-1.
\end{eqnarray}
Let us start with $m=h-1$. We have 
\begin{eqnarray}\label{eq:thm2_b_step2_sum1}
    \sum_{i \in H_{h-1}}\hatuki^2 =1-\sum_{m=1}^{h-2}\sum_{i \in H_m} \hatuki^2 - \sum_{i=p_{h-1}+1}^{d}\hatuki^2,
\end{eqnarray}
and thus we will show $ \sum_{m=1}^{h-2}\sum_{i \in H_m} \hatuki^2$ and $ \sum_{i=p_{h-1}+1}^{d}\hatuki^2$ both converge to 0. Since $\sum_{i=p_{h-1}+1}^{d}\hatuki^2 \leq \sum_{k=1}^{p_{h-1}}\sum_{i=p_{h-1}+1}^{d}\hatuki^2 = \sum_{k=p_{h-1}+1}^{d}\sum_{i=1}^{p_{h-1}}\hatuki^2 = \sum_{k=p_{h-1}+1}^{p_h}\sum_{i=1}^{p_{h-1}}\hatuki^2 + \sum_{k=p_{h}+1}^{d}\sum_{i=1}^{p_{h-1}}\hatuki^2$, it follows from $ \sum_{k=p_{h-1}+1}^{p_h}\sum_{i=1}^{p_{h-1}}\hatuki^2 = \littleo{\delfrac{h}{h-1}}$ by (\ref{eq:thm2_bi_bullet2_lasteq3}) and $\sum_{k=p_{h}+1}^{d}\sum_{i=1}^{p_{h-1}}\hatuki^2 = \bigo{\frac{d}{n\deln_{h-1}}}$ by (\ref{eq:thm2_bi_bullet2_eq2}) that 
\begin{eqnarray}\label{eq:thm2_b_step2_part2}
    \sum_{i=p_{h-1}+1}^{d}\hatuki^2 = \littleo{\delfrac{h}{h-1}}.
\end{eqnarray}
From $\sum_{m=1}^{h-2}\sum_{i \in H_m} \hatuki^2 + \sum_{i \in H_{h-1}}\hatuki^2 + \sum_{i=p_{h-1}+1}^{d}\hatuki^2 = 1$, we have 
\begin{eqnarray}\label{eq:thm2_b_step2_part1_sub1}
    \sum_{m=1}^{h-2}\sum_{i \in H_m} \hatuki^2 + \sum_{i \in H_{h-1}}\hatuki^2 \rightarrowas 1
\end{eqnarray}
 by (\ref{eq:thm2_b_step2_part2}). Also, by (\ref{eq:S1}), we have 
\begin{eqnarray*}
 \frac{1}{\lamn_k}\sumi \hatlam_i \hatuki^2 = \frac{1}{\lamn_k}\sum_{i=1}^{p_{h-2}} \hatlam_i \hatuki^2 + \frac{1}{\lamn_k}\sum_{i \in H_{h-1}} \hatlam_i \hatuki^2 + \frac{1}{\lamn_k}\sum_{i=p_{h-1}+1}^{d} \hatlam_i \hatuki^2 \rightarrowas 1~~~\mathrm{for}~k~\in H_{h-1}. 
 \end{eqnarray*}
By $\limunder{n \rightarrow \infty}  \frac{\lamn_{p_{h-1}+1}}{\lamn_k} \rightarrowas 0$, we get$\frac{1}{\lamn_k}\sum_{i=p_{h-1}+1}^{d} \hatlam_i \hatuki^2 \leq \frac{\lamn_{p_{h-1}+1}}{\lamn_k} \sum_{i=p_{h-1}+1}^{d} \hatuki^2 \rightarrowas 0$, which leads to 
 \begin{eqnarray}\label{eq:thm2_b_step2_part1_sub2}
     \sum_{m=1}^{h-2}\sum_{i \in H_m} \delfrac{m}{h-1}  \hatuki^2 + \sum_{i \in H_{h-1}} \hatuki^2 \rightarrowas 1~~~~\mathrm{for}~ k \in H_{h-1}
 \end{eqnarray}
 since $\limunder{n \rightarrow \infty} \frac{\hatlam_i}{\lamn_k} \rightarrow \delfrac{m}{h-1}$ for $i \in H_m$ and $k \in H_{h-1}$. Combining (\ref{eq:thm2_b_step2_part1_sub1}) and (\ref{eq:thm2_b_step2_part1_sub2}) with $\delfrac{m}{h-1}>1$ for $m=1, \cdots, h-2$ yields $\sum_{m=1}^{h-2}\sum_{i \in H_m} \delfrac{m}{h-1}  \hatuki^2 \rightarrowas 0$. Therefore, we have $\delfrac{h-2}{h-1} \sum_{i=1}^{p_{h-2}}\hatuki^2 \leq \sum_{i=1}^{p_{h-2}}\delfrac{m}{h-1} \hatuki^2 \rightarrowas 0$, which gives 
 \begin{eqnarray}\label{eq:thm2_b_step2_part1}
     \sum_{m=1}^{h-2}\sum_{i \in H_m} \hatuki^2 = \littleo{\delfrac{h-1}{h-2}}~~~~\mathrm{for}~k \in H_{h-1}.
 \end{eqnarray}
Hence, (\ref{eq:thm2_b_step2_main}) with $m=h-1$ follows from (\ref{eq:thm2_b_step2_sum1}), (\ref{eq:thm2_b_step2_part2}) and (\ref{eq:thm2_b_step2_part1}) and the proofs for the other $m=h-1, \cdots, 1$ are similar to $m=h-1$, and thus omitted. Lastly, we complete the proof of Theorem \ref{thm2:eigenvector} (b) by showing Step 3. 


3. Establish the strong inconsistency of the remaining sample eigenvectors, i.e. $|<\hat{U}_i, U_i>| = O(\big\{ \frac{n\lamn_i}{d} \big\}^{1/2})$ for $i > p_h$. From (\ref{eq:S1}), we have 
\begin{eqnarray*}
    \maxunder{i>p_h} \frac{n\hatlam_i}{d}\frac{d}{n\lamn_i}\hatu_{ii}^2 = \maxunder{i>p_h} \frac{\hatlam_i}{\lamn_i}\hatu_{ii}^2 \leq \meanj \frac{1}{\lamn_k}\haty_{ki}^2 \rightarrowas 1.
\end{eqnarray*}
Since $\frac{n\hatlam_i}{d} \rightarrowas \clam$ for $i > p_h$ by Theorem \ref{thm1:eigenvalue} (b), we obtain $\maxunder{i>p_h} ~\hatu_{ii}^2 = \bigo{ \frac{n\lamn_i}{d} }$, which completes the proof of Theorem \ref{thm2:eigenvector} (b).

\end{proof}

\pagebreak
\bibliographystyle{apalike}
\bibliography{references.bib} 
\end{document}